\documentclass{article}
\usepackage{amssymb}
\usepackage{amsmath, amssymb, amsthm}
\usepackage{graphicx}
\usepackage{comment}
\usepackage{subcaption}
\usepackage[utf8]{inputenc}
\usepackage{multirow}
\newcommand{\required}[1]{\section*{\hfil \sharsokol2013advanced\hfil}}

\usepackage{verbatim,url}
\usepackage{graphicx, epsfig}
\graphicspath{{LRA_figs/}}
\usepackage{bbm}
\usepackage{amssymb,amsmath,amsfonts,amsthm}
\usepackage{hyperref}
\usepackage{upref}

\newcommand{\Beq}{\begin{equation}}
\newcommand{\Eeq}{\end{equation}}
\newcommand{\be}{\begin{equation}}
\newcommand{\ee}{\end{equation}}
\newcommand{\bs}{\begin{split}}
\newcommand{\beq}{\begin{equation*}}
\newcommand{\eeq}{\end{equation*}}
\newcommand{\bal}{\begin{align}}
\newcommand{\eal}{\end{align}}

\newtheorem{theorem}{Theorem}
\newtheorem{lemma}[theorem]{Lemma}
\newtheorem{assum}[theorem]{Assumption}
\numberwithin{theorem}{section}

\usepackage{xcolor}

\newtheorem{remark}[theorem]{Remark}

\newcommand{\Nc}{\mathcal{N}}

\numberwithin{equation}{section}

\newcommand{\bel}[1]{\begin{equation}\label{#1}}

\newcommand{\eel}[1]{{\label{#1}\end{equation}}}

\usepackage{authblk} 

\title{Uncertainty Quantification for Landweber Iteration with Randomized Normal-Operator Approximation}
\author[$\dagger$]{Anuj Abhishek} 
\author[$*$]{Sean Holman}
\affil[$\dagger$]{\small Department of Mathematics, Applied Mathematics and Statistics, Case Western Reserve University, USA\par axa1828@case.edu}
\affil[$*$]{\small Department of Mathematics, University of Manchester, UK\par sean.holman@manchester.ac.uk}
\date{}
\begin{document}

\maketitle
\begin{abstract}
   \noindent  Iterative methods are extremely popular for solving linear ill-posed inverse problems. While deterministic convergence, or more precisely, semi-convergence of such methods is widely studied, the problem of quantifying uncertainty in the resulting reconstructions due to random noise in the observed data has received far less attention. In this work, we study uncertainty quantification for the Landweber iteration in linear inverse problems using the Radon transform as a motivating example. We interpret the Landweber reconstruction statistically by analyzing how uncertainty in the measured data is propagated through the reconstruction map. This perspective is closely related to generalized fiducial inference, where uncertainty about the parameter is induced by inverting the relation between the observed data and the unknown (fixed) quantity. We combine the resulting stochastic uncertainty with a theoretical bound on the regularization bias to construct confidence intervals for the reconstructed solution. The evaluation of the resulting uncertainty estimates require repeated operations with the forward operator that can become expensive in large-scale problems. To reduce this cost, we use randomized singular value decomposition to obtain a low-rank approximation of the normal operator associated with the Radon transform and incorporate this approximation in quantifying uncertainty. {{In particular, this requires that the additional randomization error introduced due to the use of such randomized techniques be included in our approach for uncertainty quantification.}} Numerical results show that the proposed method provides accurate reconstructions and reliable uncertainty estimates at substantially reduced computational cost.
\end{abstract}
\section{Introduction}
In a statistical inverse problem, one is interested in the estimation of a parameter $x^\dagger\in X\subset \mathbb{R}^n$ from noisy observations of the form, \begin{align}\label{eq:1.1}
  Y=Ax^\dagger+\epsilon
\end{align}
where $Y\in \mathbb{R}^m.$
 In this article, we assume that $A$ is a linear transform and $\epsilon$ denotes (Gaussian) random noise in the data whose distribution is known or can be estimated reliably from the observed data. In particular, we are interested in the problem of computerized tomography (CT) where the linear transform $A$ corresponds to a discrete implementation of the Radon transform acting on a function whose vectorized representation is given as $x^\dagger$. Such inverse problems are typically ill-posed making the solution very sensitive to noise in the data.  In this work, we  develop an intrinsic characterization of uncertainty in our estimation of the parameter of interest, $x^\dagger\in X$ where the estimation is done using a standard iterative method, such as the Landweber's method. Our proposed method differs from the standard frequentist paradigm, where new statistical estimators, such as kernel density estimators have been devised for such problems and various aspects of such estimators, such as minimax optimality and asymptotic confidence regions are studied, see e.g. \cite{ Arya2023Adaptive,bissantz_14,Cavalier_98,Cavalier_00,Tsybakov_1991}. Here we would like to mention that the form of such estimators is inspired by the well-known filtered back-projection (FBP) type formula. Our proposed method does not fall in the Bayesian paradigm either, wherein the solution of an inverse problem is given as the posterior distribution that is absolutely continuous with respect to the prior distribution chosen by the practitioner for the parameter to be inferred. In this case, it is the posterior distribution that is used to quantify uncertainty in the prediction, see e.g.,\cite{calvetti_somersalo_review,kaipio2006statistical,Knapik2011Bayesian,Monard_19,Monard2021Consistent,van2007bayesian}.  In contrast, our proposed methodology is closely related to the inference regime that is known among practitioners as Generalized Fiducial Inference (GFI), \cite{Hannig2009, Hannig2016,Iyer2004}.
 
To briefly explain the idea of GFI in an abstract setting, suppose one is given a `data generating equation' of the form:
 \begin{align*}
     Y=\mathcal{G}(\epsilon,x^\dagger)
 \end{align*} where $x^\dagger$ is the fixed unknown parameter and $\epsilon$ is an auxiliary random variable whose distribution function is known (or, can be estimated). For instance, the model given by \eqref{eq:1.1} above can be rewritten in the form of such a data generating equation by rewriting, $\mathcal{G}(\epsilon,x):=Ax^\dagger+\epsilon$. It is clear that any randomness in the observed data $Y$ results from the corresponding randomness in the auxiliary random variable $\epsilon$ because $x^\dagger$ is the fixed (true) parameter. For a fixed realization of the observed data $Y=y$, one then seeks an inverse map $Q_y$ which maps realizations of the random variable $\epsilon$ into the parameter space. The distribution thus induced on the parameter space is referred to as the generalized fiducial distribution, GFD, which can be used for inference. Thus, the central idea in GFI is to transfer the randomness from the data to the parameter space using an inverse of a data-generating equation without the use of Bayes’ theorem. In particular, unlike Bayesian inference, randomness in the parameter space is not introduced by any (subjective) choice of the prior distribution but rather the randomness present in the data generating equation through the auxiliary variable $\epsilon$ is used to introduce a distribution on the parameter space. 
 \par For ill-posed inverse problems, however, the exact inverse associated with the data-generating equation is generally unstable and is therefore unsuitable for uncertainty quantification. Our main idea is to replace this exact inversion by a regularized iterative inversion. In particular, for the Landweber method, early stopping introduces a form of regularization. Thus the $K-$th iterate in such a method can be thought of as a regularized inversion map from the data space to the parameter space, and we use this map to pushforward the distribution of the observational noise, see our discussion in section \ref{sec:fid_inf} below. Note that, this idea produces for every iteration $k$ a data-dependent distribution over possible reconstructions, thereby placing the classical Landweber method in a probabilistic iterative framework (PIM), see e.g \cite{Ipsen_21}, where the calibration of such methods is also studied.
 
 Our goal in this article is different, where we seek to use a distributional output at the (stopping) $K-$th iteration to capture the uncertainty in the reconstruction. The stopping index, $K$, consequently plays a dual role: it determines the regularization of the reconstruction and simultaneously controls the propagated uncertainty. As we show in section \ref{sec:err_dist}, this allows us to study explicitly the distribution of the reconstruction error and its decomposition into a bias and a variance term. In such an analysis, the role of the stopping index $K$ becomes clearer. Early stopping has the effect of controlling the variance term and thus allowing noise amplification to be controlled. However, early stopping also introduces a bias in the reconstruction which needs to be carefully estimated. We show in this work how to combine the generalized fiducial distribution generated by the probabilistic iteration with bounds on this bias and an appropriate stopping criterion in order to construct confidence regions for the unknown parameter, $x^\dagger$.  On the other hand, in a series of recent articles \cite{Katsevich2025,Katsevich2026}, reconstruction from noisy discrete generalized Radon data at native resolution has been analyzed and it was shown that the local reconstruction error converges to a Gaussian random field with an explicitly characterized covariance structure which builds on earlier articles, \cite{Ab_Ka_We_1,Ab_Ka_We_2}. Our approach in this work is complementary in that we study uncertainty generated by an iterative regularization method, where the propagated observational noise is combined with an explicit bound on the regularization bias introduced by early stopping, so as to construct confidence regions.
 \par {{Finally, we consider the applicability of the analysis developed in sections \ref{sec:fid_inf} -\ref{sec:3.5} for large scale ill-posed inverse problems. Herein, the fact that calculation of the propagated covariance matrix required for uncertainty quantification requires us to iterate over a polynomial in the normal operator, $A^TA$, makes the computation  excessively expensive and susceptible to noise in the data, see our discussion at the beginning of  section \ref{sec:3.6}. Through the use of randomized SVD (rSVD) approximation of the normal operator, we address both of these problems. First off, working with a truncated rank $r-$ approximation of the normal operator reduces the computational overhead. Secondly, it introduces an additional level of regularization that makes the computation less susceptible to the level of noise in the data. A similar philosophy of introducing additional regularization through the use of randomized SVD techniques also motivated the authors of \cite{ito2020} where rSVD was combined with classical regularization methods such as truncated SVD (TSVD) and Tikhonov regularization. More recently, randomized techniques have been combined with Krylov subspace methods for computing regularized solutions to large-scale inverse problems, see \cite{ChungGazzola2025}. In the context of the present work, introducing this additional randomization in the method calls for carefully revisiting the bias and variance estimates, which we do in subsections \ref{sec:3.7} and \ref{sec:3.6.2} where we first devise a new stopping rule based on bias and variance approximations arising from our use of randomized SVD and subsequently construct confidence regions for which we also provide lower bounds on frequentist coverage. }}
 \par The organization of the article is as follows: in section \ref{sec:2}  we describe the mathematical set up for the problem, and in section \ref{sec:3} describe in detail our proposed method. In particular, subsection \ref{sec:3.6.2} illustrates the method of constructing confidence regions used in the present work and provides theoretical justification for the proposed method. In section \ref{sec:4}, we describe the numerical simulations that were run to check the validity of the proposed approach. Section \ref{sec:app} contains some auxiliary results and discussion needed for our core analysis.

\section {Observation model}\label{sec:2}
Let $f: \mathbb R^2 \to \mathbb R$ denote the unknown image (function).  We observe noisy samples of its Radon transform on a fixed design grid of angles and detector locations. More precisely, let $\{\vartheta_i\}_{i=1}^{n_\vartheta}\subset [0,\pi)$ be projection angles and $\{u_j\}_{j=1}^{n_u}\subset \mathbb R$ be the detector locations. The Radon transform of $f$ is then defined by the line integral:
\[
(Rf)(\vartheta_i,u_j)=\int_{\{x\in\mathbb{R}^2:\,x\cdot\theta_i=u_j\}}f(x)\,d\ell(x),
\quad
\theta_i=(\cos\vartheta_i,\sin\vartheta_i).
\]
In this work we assume that the data is corrupted by $\textrm{i.i.d.}$\,(independently and identically distributed) Gaussian noise. Thus, the measured sinogram  data is modeled by:
$$ Y_{ij} = (Rf)(\vartheta_i,u_j) + \varepsilon_{ij},
\qquad i=1,\dots,n_\vartheta,\quad j=1,\dots,n_u, $$
where the noise variables $\varepsilon_{ij}$ are $\textrm{i.i.d}$\,Gaussian with
$ \mathbb E[\varepsilon_{ij}] = 0,$ and $\mathbb E[\varepsilon_{ij}^2] = \sigma^2.$ 

In this article, we will present our results in a discrete setting. To that end, we  discretize $f$ in a finite-dimensional basis. For example, we represent $f$ by pixel values on a Cartesian grid and flatten it into a column vector $x^\dagger\in\mathbb R^n$. Using a standard discretization scheme, the discretized Radon transform at $(\vartheta_i,u_j)$ can be written as a weighted sum of pixel values,
$$(Rf)(\vartheta_i,u_j) \approx \sum_{k=1}^n a_{(i,j),k} x^\dagger_k,
$$
where the coefficients $a_{(i,j),k}$ encode the intersection length of the ray parametrized by $(\vartheta_i,u_j)$ with pixel $k$. Next we flatten the observed samples into a single vector
$$ y = \mathrm{vec}(Y_{ij}) \in \mathbb R^m,
\qquad m=n_\vartheta n_u,$$ in such a way, so as to write $$ y = A x^\dagger + \varepsilon,
\qquad \varepsilon \sim N(0,\sigma^2 I_m). $$ Here the matrix $A\in\mathbb R^{m\times n}$ is such that its $(i,j)$-row equals $(a_{(i,j),1},\dots,a_{(i,j),n})$ where the index $(i,j)$ corresponds to some index $p\in 1,\dots,m$ in the flattened vector. We will use this formalism in the later sections when we describe the probabilistic iterative method. However, for the numerical simulations used in this work, we do not explicitly construct such matrices $A$ and rather work with MATLAB's in-built `radon' and `iradon' implementations to do the computations in a matrix-free manner. Throughout this work, we will assume that the variance $\sigma^2$ is given. In practice, it can be estimated from the observed data using a variance based estimator, see \cite{bissantz_14, bissantz_05}.

\section{Proposed Method}\label{sec:3}
\subsection{Landweber iteration in Discrete Setting}\label{sec:3.1}
In this section, we begin with a brief survey of results on Landweber method \cite{Landweber1951}, which is a classical iterative technique used to solve inverse problems.
Given the observation model \eqref{eq:1.1}, in the classical Landweber method we generate a sequence of iterates starting from an initial guess $x_0$ according to the equation
\begin{equation}\label{eq:3.1}
    x^{(k+1)} = x^{(k)} - \omega A^T (A x^{(k)} - y)
\end{equation}
where $0<\omega<2/\lVert A^TA\rVert$ is a relaxation parameter and the norm used is the spectral norm. 
Note that by simply rearranging the terms we can write \eqref{eq:3.1} in the form
\begin{equation}\label{eq:3.2}
     x^{(k+1)} = G(x^{(k)})+f
\end{equation}
where
$G =  (I - \omega A^T A) \quad \text{and} \quad f=\omega A^Ty.$ Note that the choice of $\omega$ ensures that the map $G$ is a contraction and the spectral norm $\rho(G)<1$ whenever the matrix $A$ is full column rank.
The Landweber iteration has been extensively studied as an iterative regularization method for ill-posed inverse problems. For such problems, with a proper choice of the relaxation parameter as above, the iterates converge to the generalized solution for the exact data. However, for noisy data, the Landweber iteration exhibits a phenomenon called `semi-convergence' which requires the iteration to be stopped appropriately using a suitable stopping rule. In this sense the stopping iteration index can be thought of as a regularization parameter, \cite[section 6.1]{Engl_book}. In \cite{hanke_91}, Landweber-type methods for linear inverse problems were analyzed and it was shown that with appropriately chosen stopping rules, they serve as regularization methods and achieve optimal accuracy under suitable smoothness assumptions. Going beyond linear problems,   Landweber iteration was applied to non-linear problems in \cite{hanke_95} and a discrepancy principle based stopping rule was analyzed. Furthermore, the Landweber iterations were successfully applied to the study of inverse scattering problem in \cite{HankeHettlichScherzer1995} and a modified Landweber method was used for parameter estimation problem in \cite{Scherzer1998ModifiedLandweber}. They still continue to be an active object of study where they have been more recently applied to problems arising in cosmological X-Ray tomography, \cite{ChungOniskWang2025}, while at the same time modified versions of the original Landweber method continue to be developed, \cite{Aspri2020,AspriScherzer2023}. Landweber reconstruction has also been studied specifically for X-ray type transforms. In \cite{Holman2018},  the attenuated geodesic X-ray transform was analyzed and it was shown how microlocal instability of the forward operator leads to characteristic artifacts in Landweber reconstructions.

For the purpose of this present work, however, we will use the classical Landweber iteration \eqref{eq:3.1} to develop our ideas on `generalized' fiducial inference for iterative methods. Based on the discussion above that the K-th iterate can be thought of as a regularized inverse map for the given problem we 
rewrite the Landweber iteration in the following manner. First define,
$T_K = \omega\sum_{s=0}^{K-1} G^s A^T.$ Then starting with $x_0=0$, it is easy to see that, \begin{equation} \label{eq:LW_est}
  x^K(y) = T_K y .
\end{equation}
 %The stopping rule we use here is from \cite[eq. 6]{BC19} which is given by the so-called risk-minimization based discrepancy principle. We will denote this stopping rule by $K(\sigma,y_{\sigma})$ using the same notation as in \cite{BC19}. We also note that $K(\sigma,y_{\sigma})$ is always finite, see \cite[Theorem 4]{BC19}. We will denote the `Landweber solution' got by stopping the iteration at the optimal value $K(\sigma, y_{\sigma})$ by
%$x_{\sigma}^{K(\sigma, y_{\sigma})}$. Then, from \cite[Theorem 5]{BC19} we have:
%\begin{equation}\label{eq:convergence1}
%    \lim_{\sigma \to 0} \mathbb{E} \| x_{\sigma}^{K(\sigma, y_{\sigma})} - x^{\dagger} \|^2 = 0.
%\end{equation}
%\ac{I wonder if we might use some other more easily implementable stopping rule and still have something akin to \eqref{eq:convergence1} hold. I will try to use the standard discrepancy principle.}
\subsection{Fiducial inference for the inverse problem }\label{sec:fid_inf}
We now formulate the fiducial construction associated with the observation model \eqref{eq:1.1} and then describe its regularized counterpart for the Landweber method in the next subsection.
We recall the data generating equation,
\begin{equation} \label{eq:data_gen}
    Y = A x^{\dagger} + \epsilon, \quad \text{where} \quad \epsilon \sim \mathcal{N}(0, \sigma^2 I).
\end{equation}
Following \cite{Hannig_16}, we recognize this equation to be of the type
\[
Y =\mathcal{G}(U,\theta)=A\theta+\sigma U,
\qquad U\sim\mathcal N(0,I),\quad \theta:=x^\dagger\in\mathbb R^{n}.
\] Here, $U \sim N(0,I_m)$ is an auxiliary random variable with known distribution and $x^\dagger\in\mathbb R^{n}$ is the fixed unknown parameter. Following the presentation in \cite[Example 1]{Hannig_16}, we will first evaluate the exact Fiducial distribution on the parameter space to motivate our idea of constructing a generalized confidence distribution in the subsequent sections. In fiducial inference, for a fixed realization, $Y=y$ one seeks to invert the data generating equation with respect to the parameter while retaining the randomness of the auxiliary variable, $U$. In the present linear setting, assuming that the matrix $A$ has full column rank, the corresponding inversion map   for $x^{\dagger}$ in \eqref{eq:data_gen} gives the Fiducial inversion map:
$$Q_y(U) = (A^T A)^{-1}A^T (y - \sigma U)$$ and therefore the Fiducial distribution is the distribution of the random variable $x_{\sigma}|y:=Q_y(U)$ and as such:
\begin{align}
 x_{\sigma}|y\sim \Nc((A^TA)^{-1}A^Ty,\sigma^2 (A^TA)^{-1}).
\end{align}
For a well-conditioned problem, this provides a direct fiducial characterization of uncertainty in the estimation of $x^\dagger.$ In the ill-posed inverse problems considered here, however, the exact inversion map is unstable, since the mean as well as covariance involves computing $(A^TA)^{-1}$ for an ill-conditioned matrix, $A$. This motivates the principal idea proposed in this work. Rather than using the unstable exact fiducial inversion map, we replace it by the regularized inverse generated by a finite number of Landweber iterations. 

 Recall that, starting from $x_0=0$ the $K$-th Landweber iterate can be written as
\[
x^K= T_K(y):=\omega\,\sum_{s=0}^{K-1}G^s A^T y ,
\qquad 
G:=I-\omega A^T A.
\]
Thus, we have the following regularized {fiducial inversion} map
\[
Q^{K}_y(u):= T_K\big(y-\sigma u\big)
=\omega\,\sum_{s=0}^{K-1}G^s A^T\big(y-\sigma u\big),
\qquad u\sim\mathcal N(0,I_m).
\]
% \seanc{Maybe change notation from $Q^{K}_y$ to $Q^K_y$ or $Q^{reg}_{y,K}$.}
% \ac{Changed to $Q^K_y$}
Note that this can be interpreted as pushing forward the known law of $U$ through $Q^{K}_y$ to produce a data-dependent distribution on $x$ at iteration $K$. In this sense the standard Landweber iterative method can be thought of as belonging to the class of probabilistic iterative methods of \cite{Ipsen_21}; however we note that unlike in \cite{Ipsen_21}, the distributional output at step $K$ does not depend upon any initial `belief distribution'. A simple calculation shows that the induced (pushforwarded) fiducial law for $x^{\text{reg}}|y$ is the Gaussian distribution:
\begin{align}\label{eq:3.2.1}
x^{\text{reg}}| y \sim \mathcal N\big(x^K,\Gamma_K\big),
\qquad
x^K = \omega\,\sum_{s=0}^{K-1}G^sA^T y,\qquad
\Gamma_K=\sigma^2\,\omega^2\,\left(\sum_{s=0}^{K-1}G^s\right)A^T A\,\left(\sum_{s=0}^{K-1}G^s\right)^T.
\end{align}
% \seanc{I changed $\Sigma^K_\sigma$ to $\Gamma_K$ in \eqref{eq:3.2.1} to be consistent with the later notation. Does $\Sigma_K$ appear anywhere without $\sigma$? Generally, it might be useful to use $\Gamma_K$ throughout if possible.}
% \ac{Good point! I'll change this throughout}
The above construction should be viewed as a regularized analogue of the exact fiducial inversion. The exact fiducial map solves the data generating equation directly, whereas the finite-step Landweber map replaces the unstable inverse by an iterative regularization operator. Thus, for each \(K\), the auxiliary randomness is propagated through the same regularized inverse used to compute the reconstruction.
The mean of the resulting distribution is precisely the classical \(K\)-th Landweber iterate. The covariance \(\Gamma_K\) quantifies the propagation of the observational noise through the finite-step inverse. The iteration index \(K\) therefore controls both the regularization of the reconstruction and the amount of uncertainty transferred from the data space to the parameter space.

Unlike the exact fiducial distribution, however, the regularized distribution is affected by the bias introduced through early stopping. Consequently,using \eqref{eq:3.2.1} for uncertainty quantification requires us to estimate the bias. This is carried out in the following subsection by evaluating the distribution of the reconstruction error.

\subsection{Distribution of the error in the reconstruction}\label{sec:err_dist}

We now study the error associated with the mean of the regularized fiducial distribution obtained in the previous subsection. Recall that the mean of the distribution in \eqref{eq:3.2.1} is the usual \(K\)-th Landweber iterate $x^K = T_K y.$ If we use the mean of the above distribution as the estimated parameter, then the error in the estimation is distributed according to the following theorem.
% \seanc{I've reworked the numbering so that theorems, lemmas, remarks and propositions are all numbered together.}
% \seanc{I think the expectation in \eqref{errordist} should be changed to $G^K x^\dagger$. It is simpler and that matches with the proof and how it's used later.}
% \ac{changed it.}
\begin{theorem}\label{thm:3.1}
   The distribution of the error $x^{\dagger} - x^K$ is given by

\begin{equation}\label{errordist}
    x^{\dagger} - x^K \sim \mathcal{N}\left( {G^K x^{\dagger}},  \Gamma_K \right),
\end{equation}
where
\begin{equation}\label{eq:Cov_k}
    \Gamma_K =  \sigma^2\omega^2 \left (\sum_{k=0}^{K-1} G^k\right) A^T A \left (\sum_{j=0}^{K-1} G^j\right).
\end{equation} 
\end{theorem}
\begin{remark}In particular, this shows that the error follows a normal distribution with a mean given by the bias term
\begin{equation} \label{bias1}
    \mathbb{E}[x^{\dagger} - x^K] = G^K x^\dagger=\left(I - \sum_{k=0}^{K-1} G^k \omega A^T A \right) x^{\dagger},
\end{equation}
and a covariance given by
\begin{equation} \label{cov1}
    {\rm Cov}(x^{\dagger} - x^K) = \Gamma_K.
\end{equation}
\end{remark}
% \seanc{I've changed the proof a bit and made it shorter.}
\begin{proof}

We know that the Landweber iterates follow the form
\begin{equation}
    x^{(k+1)} = G x^{(k)} + f,
\end{equation}
where 
\begin{equation}
    G = I - \omega A^T A, \quad f = \omega A^T y.
\end{equation}
Recall that the data is given by, 
$ y = A x^{\dagger} + \epsilon, \,\text{where } \epsilon \sim \mathcal{N}(0, \sigma^2 I).
$
Substituting this into the formula for $f$ gives
\begin{equation}
    f = \omega A^T(A x^{\dagger} + \epsilon) = \omega A^T A x^{\dagger} + \omega A^T \epsilon.
\end{equation}
Assuming the iterations to start with $x_0=0$, after $K$ iterations we have
\begin{align}
    x^K =  \sum_{k=0}^{K-1} G^k f= \left(\sum_{k=0}^{K-1} G^k \omega A^T A \right) x^{\dagger} + \sum_{k=0}^{K-1} G^k \omega A^T \epsilon .
\end{align}
The first term, 
$\left(\sum_{k=0}^{K-1} G^k \omega A^T A \right) x^{\dagger}=(I-G^K)x^\dagger,
$
is deterministic, while the second term, 
$\sum_{k=0}^{K-1} G^k \omega A^T \epsilon,
$
is a linear transformation of an independent Gaussian variable $\epsilon \sim \mathcal{N}(0,\sigma^2 I)$. Since Gaussian distributions remain Gaussian under linear transformations and sums, we get that
\begin{equation}
    \sum_{k=0}^{K-1} G^k \omega A^T \epsilon \sim \mathcal{N}\!\left(0, \; \sigma^2 \omega^2 \left (\sum_{k=0}^{K-1} G^k\right) A^T A \left (\sum_{j=0}^{K-1} G^j\right)\right).
\end{equation}
Thus, we obtain
\begin{equation}
    x^K \sim \mathcal{N}\left ((I-G^K) x^{\dagger}, \Gamma_K\right),
\end{equation}
where $\Gamma_K$ is given in \eqref{eq:Cov_k}. Since $x^\dagger$ is deterministic, we therefore obtain 

% Now we compute the distribution of $x^{\dagger} - x^K$.
% Note that,
% \begin{align}
%     x^{\dagger} - x^K &= x^{\dagger} - \left(\sum_{k=0}^{K-1} G^k \omega A^T A\right) x^{\dagger} - \sum_{k=0}^{K-1} G^k \omega A^T \epsilon  \\
%     &=\left(I - \sum_{k=0}^{K-1} G^k \omega A^T A\right) x^{\dagger} - \sum_{k=0}^{K-1} G^k \omega A^T \epsilon\label{eq:4.16}.
% \end{align}
% Thus,
\begin{equation}
    x^{\dagger} - x^K \sim \mathcal{N}\left(G^K x^{\dagger}, \Gamma_K\right).
\end{equation}

\end{proof}
%\begin{remark}
  Theorem \ref{thm:3.1} shows that the reconstruction error consists of two  different components. The stochastic component is induced by the observational noise and its distribution is characterized explicitly by the covariance matrix $\Gamma_K$. The deterministic component, $G^K x^\dagger,$ is the regularization bias introduced by terminating the Landweber iteration after finitely many steps. As a result, the construction of confidence regions later for quantifying uncertainty in the estimation requires a quantitative bound on this bias term along with the stochastic error term. We will first obtain such a bound on the bias term under a standard source condition.
%\end{remark} 

\subsection{Source condition and regularization bias}\label{sec:src}
We impose a standard source condition on the exact solution similar to the ones introduced in  \cite{ hanke_91, hanke_95}.
In fact, since our setting is in finite dimensions, we will assume the following variant of the source condition from \cite{gerth_19,hanke_91}.
\begin{assum}\label{src_assum}
   There exist $\nu>0$, $R>0$, and a vector $w\in\mathbb{R}^n$ with $\|w\|_2\leq M$ such that
\[x^\dagger = (A^T A)^\nu w.
\] 
\end{assum}
\noindent The source condition restricts the class of admissible exact solutions by prescribing their regularity relative to the forward operator. Such a restriction is standard in the analysis of ill-posed inverse problems, since without additional regularity assumptions the convergence of regularized solutions may be arbitrarily slow. The parameter $\nu$ quantifies the smoothness of $x^\dagger$ with respect to $A^T A$ and, consequently, determines the rate at which the regularization error can decay. In the present setting, we use this condition to obtain a quantitative bound on the deterministic bias introduced by terminating the Landweber iteration after finitely many steps. We point the reader to \cite{gerth_19, hanke_91} for a more elaborate discussion of using such source conditions in inverse problems of the kind considered here.

 Let us now recall the following result from \cite[Theorem 3.2 and example 3.1]{hanke_91}. Under the source condition assumed above, the deterministic regularization error decays algebraically with the number of Landweber iterations, with the decay rate determined by the parameter $\nu$. The lemma below gives a precise estimate for an upper bound on this bias term. %, i.e.
%$\|b_k\|\leq C_{\nu}M k^{-\nu}$ where $C_{\nu}$ is a constant that depends on $\nu$ and $\|A\|$, see also \cite[Theorem 3.1]{hanke_95} for a similar result. In fact, the Lemma below gives a more precise estimate for an upper bound on this bias term.  
% \seanc{Reworded the above paragraph slightly to avoid using $b_K$ before it is introduced in the statement of the next Lemma. Adding a remark after the statement of the Lemma with removed content.}
\begin{lemma}
    \label{thm:3.2}
Let us assume the source condition such that for some $\nu>0$,  $x^\dagger=(A^T A)^\nu w,\,
    \|w\|_2\leq M.$ Let us denote an eigendecomposition of, $A^TA$ as 
    \[ A^T A =V\operatorname{diag}(s_1^2,\ldots,s_n^2)V^T,\] where \(s_1,\ldots,s_n\) are the singular values of $A$. Let $G=I-\omega A^T A$ as above. Then the regularization bias after $K$ iterations satisfies
    \[\|G^Kx^\dagger\|_2 \leq  M\max_{1\leq i\leq n}
    \left|1-\omega s_i^2\right|^K(s_i^2)^\nu:=b_K(\nu,A).\] Moreover, if $\omega=\frac{1}{2s_{\max}^2},
    \, s_{\max}:=\max_i s_i,$ then \begin{align} \label{bkest}
    b_K(\nu,A) \leq
    2^\nu s_{\max}^{2\nu}M
    \frac{\nu^\nu K^K}{(\nu+K)^{\nu+K}}:=b_K^{ub}(\nu,A).
\end{align}
where $b_K^{ub}(\nu,A)$ denotes a global upper bound on the bias.
\end{lemma}
\begin{remark}
The estimate \eqref{bkest} also shows that $b_K(\nu,A) \leq C_{\nu}M K^{-\nu}$ where $C_{\nu}$ is a constant that depends on $\nu$ and $\|A\|$. See also \cite[Theorem 3.1]{hanke_95} for a similar result. 
\end{remark}
\begin{proof}
Note that, \[
G^K(A^T A)^\nu
=V\operatorname{diag}_n\left((1-\omega s_i^2)^K(s_i^2)^\nu\right) V^T.\] Thus it follows that, \begin{align*}
    \|G^Kx^\dagger\|_2
    &=
    \|G^K(A^T A)^\nu w\|_2 \leq
    \|G^K(A^T A)^\nu\|_2\,\|w\|_2 \\
    &\leq
    M\max_{1\leq i\leq n}
    \left|1-\omega s_i^2\right|^K(s_i^2)^\nu,
\end{align*} Now if we choose, $\omega=\frac{1}{2s_{\max}^2}$ and let $ t_i=\omega s_i^2$ then $0\leq t_i\leq 1/2$. Thus $s_i^{2\nu}= \left(\frac{t_i}{\omega}\right)^\nu =2^\nu s_{\max}^{2\nu}t_i^\nu.$ From this it follows that, 
\[ b_K(\nu,A)
    \leq
    2^\nu s_{\max}^{2\nu}M
    \max_{0\leq t\leq 1/2}
    (1-t)^Kt^\nu\leq 2^\nu s_{max}^{2\nu} M \frac{\nu^\nu K^K}{(\nu+K)^{\nu+K}}.\]
\end{proof}

\subsection{Construction of a guaranteed confidence region}\label{sec:3.5}
{{Recall that a $(1-\alpha)$-level confidence region for an unknown
parameter $x^\dagger \in \mathbb{R}^n$ is a data-dependent random set
$C(Y)\subset \mathbb{R}^n$ satisfying
\[\mathbb{P}_{x^\dagger}\big(x^\dagger\in C(Y)\big) \geq 1-\alpha .
\]
%The probability above is taken with respect to repeated realizations
%of the data under the fixed true parameter $x^\dagger$.
% \seanc{I'd remove the last sentence (and remove ``furthermore" from the next one).} 
% \ac{done.}
$\mathbb{P}_{x^\dagger}(x^\dagger\in C(Y))$ is referred to as the coverage probability
of the confidence region; see, for example,
\cite[Sections 9.1 and 9.3.1]{CB_book}. Thus, the randomness in the construction of the confidence region
arises from the randomness of the observed data $Y$. Once a particular
realization $Y=y$ is observed, $C(y)$ is the corresponding realized
confidence region.}}
We first construct a confidence region under the assumption that the covariance matrix $\sigma^2\Sigma_K$ can be evaluated exactly for some stopping index, $K$, whose choice will be explained below. To that end, define $\delta_K(\nu,A):=\frac{b_K^{ub}(\nu,A)}{\sqrt{\lambda_{\min}(\Gamma_K)}},
$ where $\lambda_{\min}(\Gamma_K)$ denotes the smallest eigenvalue of  $\Gamma_K$. We denote by $\chi_n^2$ a chi-squared random variable with $n$ degrees of freedom, and by $\chi^2_{n,1-\alpha}$ its $(1-\alpha)$-quantile, i.e., \[     \mathbb{P}\left(  \chi_n^2 \leq \chi^2_{n,1-\alpha}     \right)=1-\alpha . \] 
%For ease of notation, we will also denote $\sigma^2\Sigma_K:=\Gamma_K$ from this point on. 
Then we have the following theorem.
\begin{theorem}\label{thm:3.3}
Suppose that assumptions of Lemma \ref{thm:3.2} hold. For $0<\alpha<1$, define \begin{align}
C_K(y)
:=
\left\{
x:\ (x-x^K_\sigma)^T(\Gamma_K)^{-1}(x-x^K_\sigma)
\le
\Big(\sqrt{\chi^2_{n,1-\alpha}}+\delta_K(\nu,A)\Big)^2
\right\}. 
\end{align}
 Then, \[
\mathbb{P}_{x^\dagger}\left(x^\dagger\in  C_K(Y)\right)\ \geq\ 1-\alpha.
\]
Hence, $C_K(Y)$ is a $(1-\alpha)$-level confidence region for
$x^\dagger$.
\end{theorem}
\begin{proof}
Recall that \eqref{errordist} implies that 
 \[x^{\dagger}-x_{\sigma}^K=G^Kx^\dagger+Z_K\] where $Z_K\sim N(0,\Gamma_K)$. Thus, $G^Kx^\dagger$ is the deterministic regularization bias,
whereas $Z_K$ is random reconstruction error. In order to construct confidence regions we look at the distribution of the random variable given by $(x^\dagger-x_{\sigma}^K)^T\Gamma_K^{-1}(x^\dagger-x_{\sigma}^K):=\|(x^\dagger-x_{\sigma}^K)\|^2_{\Gamma^{-1}_K}$. 
By the triangle inequality,  \[
\|x^\dagger-x^k_{\sigma}\|_{{\Gamma_K}^{-1}}
\leq
\|Z_K\|_{{\Gamma_K}^{-1}} + \|G^K x^\dagger\|_{{\Gamma_K}^{-1}}.
\]  Since $Z_K\sim N(0,\Gamma_K)$, it follows that
$\Gamma_K^{-1/2}Z_K\sim N(0,I_n).
$
Thus
$\|Z_K\|_{(\Gamma_K)^{-1}}^2
\sim
\chi_n^2.
$ Therefore, by the definition of $\chi^2_{n,1-\alpha}$,
\[\mathbb{P}\left(
\|Z_K\|_{\Gamma_K^{-1}}\leq
\sqrt{\chi^2_{n,1-\alpha}}
\right)
=1-\alpha.
\]
Also, from Lemma \ref{thm:3.2} and from the definition of $\delta_K(\nu, A)$ we have, 
$\|G^Kx^\dagger\|_{(\Gamma_K)^{-1}}
\leq
\delta_K(\nu,A).
$ Thus, on the event
$\left\{
\|Z_K\|_{\Gamma_K^{-1}}\leq
\sqrt{\chi^2_{n,1-\alpha}}\right\},
$
we have
\[
\|x^\dagger-x_\sigma^K\|_{\Gamma_K^{-1}}
\leq\sqrt{\chi^2_{n,1-\alpha}}+\delta_K(\nu,A).
\]
Equivalently,
\[
(x^\dagger-x_\sigma^K)^T
\Gamma_K^{-1}
(x^\dagger-x_\sigma^K)
\leq\Big(
\sqrt{\chi^2_{n,1-\alpha}}
+\delta_K(\nu,A)
\Big)^2,
\]
which implies that $x^\dagger\in C_K(Y)$. Hence,
\[\mathbb{P}_{x^\dagger}\left(x^\dagger\in C_K(Y)
\right)\geq \mathbb{P}\left(
\|Z_K\|_{\Gamma_K^{-1}}\leq
\sqrt{\chi^2_{n,1-\alpha}}
\right)
=1-\alpha.
\]
\end{proof}
% \seanc{The result above looks correct, but I think the proof needs a bit more. First, the random vector $Z_K$ should be defined. Second, the meaning of $\|b_K \|_{\Lambda_K^{-1}}$ is not totally clear since $b_K$ is a constant, not a vector. Actually, after reading rest of the section I'm not sure I understand the notation $b_K$ and, as far as I can see, it's only been defined as a scalar prior to this. Maybe you mean to define the vector $b_K = G^K x^\dagger$?}
% \ac{Fixed this! You are right, the notation was sloppy and I meant to use $G^K x^\dagger$ for the bias vector and use the estimate on its norm as above.}
% \seanc{Below, I suggest using $\Gamma$ rather than $\Sigma$ for consistency. Also, I think we should explain the ``pixel-wise" UQ before we describe the technique for finding K below.}
% \ac{done!}
While Theorem~\ref{thm:3.3} provides a joint ellipsoidal confidence
region for the unknown parameter $x^\dagger$, in imaging applications
it is often more useful to quantify uncertainty separately at each
pixel. We therefore also consider marginal, or pixelwise, confidence
regions obtained from the distribution of each coordinate of the
reconstruction error. Recall that $x^\dagger-x_\sigma^K=G^Kx^\dagger+Z_K, \, Z_K\sim N(0,\Gamma_K).$
Hence, for each pixel $j=1,\ldots,n$,
\[(Z_K)_j \sim N\left(0,(\Gamma_K)_{jj}\right).\] Also, trivially, $|(G^Kx^\dagger)_j|\leq
b_K^{ub}(\nu,A)$. From this, we obtain
\begin{align}
 \label{pixelUQ}   \mathbb{P}
\left(x^\dagger_j \in \bigg\{(x_\sigma^K)_j \pm\,
\Bigl( z_{1-\alpha/2}\,\sqrt{(\Gamma_K)_{jj}} + b_K^{ub}(\nu,A) \Bigr)\bigg\}\right)\geq 1-\alpha.
\end{align}
We now explain the mechanism by which the stopping index, $K$ is evaluated. This is done by balancing the regularization bias with the stochastic uncertainty captured by the variance term. In fact, let
$v_k(\sigma,A):=
\max_{1\leq j\leq n}
\sqrt{(\Gamma_k)_{jj}},$ and define
\begin{align}\label{exact_stop}
K:=K(\sigma)=\inf
\left\{
k\geq 1:
b_k^{ub}(\nu,A)
\leq
c_0 v_k(\sigma,A)
\right\},
\end{align} for some fixed constant $c_0>0$ and where $b_k^{ub}(\nu,A)$ and $\Gamma_k$ are computed at any iteration $k$ of the Landweber method by an obvious modification of the formulas for calculating $\Gamma_K$ and  {{$b_{K}^{ub}(\nu,A)$ from Theorem \ref{thm:3.1} and Lemma \ref{thm:3.2} respectively.}} Since the stopping index $K(\sigma)$ depends on the noise level and the forward problem, but not on the realization of the data $Y=y$. Hence Theorem \ref{thm:3.3} applies in this case with $K=K(\sigma)$ as chosen. 

{Computation of the early-stopping index $K$ using \eqref{exact_stop} as well as construction of confidence regions using \eqref{pixelUQ} require knowledge of the diagonal entries, $\gamma_{K,j}:=(\Gamma_K)_{jj}$. We may estimate these diagonal entries using tools from randomized linear algebra. In particular, we use a Monte Carlo method for estimating the diagonal entries of a real-symmetric matrix as in \cite[section 5]{hallman2023}. We describe this construction here briefly. {{Recall that a Rademacher random variable is a random variable that takes value $\pm 1$ with equal probability, $1/2$. A Rademacher random vector is a vector whose every element is an independent Rademacher random variable.}} Let $w_1,\ldots,w_N\in\mathbb{R}^n
$ be independent Rademacher vectors and define
\begin{align}\label{mc_est}\widehat\gamma_{K,j}
=\frac1N\sum_{\ell=1}^N
\left(\Gamma_Kw_\ell\right)_j(w_\ell)_j,\quad
\gamma_{K,j}:=(\Gamma_K)_{jj}.
\end{align} By \cite[Theorem 3.5 (1)]{hallman2023}, it is known that $\widehat\gamma_{K,j}$ is an unbiased estimator of $\gamma_{K,j}$. Furthermore from \cite[Corollary 5.4]{hallman2023} it is known that for prescribed tolerance $0<\epsilon<1$ and failure probability $0<\delta<1$, \begin{align}
P\left(|\widehat\gamma_{K,j}-\gamma_{K,j}|
\leq\varepsilon\gamma_{K,j}\right)\geq 1-\delta,
\end{align} provided the number of Rademacher probes $
N\geq
\frac{2}{\varepsilon^2}
\left[\frac{\|(\Gamma_K)_{:j}\|_2^2-\gamma_{K,j}^2
}{\gamma_{K,j}^2}
\right]\log\left(\frac{2}{\delta}\right)
$ where $(\Gamma_K)_{:j}$ denotes the $j-th$ row of the matrix, see also \cite[Definition 2.6]{hallman2023}. Finally, on the event
$|\widehat\gamma_{K,j}-\gamma_{K,j}|
\leq
\varepsilon\gamma_{K,j},$
we have $(1-\varepsilon)\gamma_{K,j}
\leq\widehat\gamma_{K,j},
$ and hence
$\gamma_{K,j}
\leq
\frac{\widehat\gamma_{K,j}}{1-\varepsilon}.
$ Thus if we define an 
inflated pixelwise confidence region
$$
\widehat C_{K,j}(Y)
=\left\{
(x_\sigma^K)_j\pm
\left[z_{1-\alpha/2}
\sqrt{\frac{\widehat\gamma_{K,j}}{1-\varepsilon}
}+b_K^{ub}(\nu,A)\right]\right\}.
$$ then we have
\begin{align}\label{eq:3.26}
P\left(x_j^\dagger\in\widehat C_{K,j}(Y)\right)\geq(1-\alpha)(1-\delta).
\end{align}
While \eqref{eq:3.26} is a recipe for producing valid pixelwise confidence intervals, its utility for doing practical uncertainty quantification is somewhat limited. To understand why this may be the case, let us recall that the exact propagated
covariance matrix $\Gamma_k$ for any iteration $k\leq K$ can be written in terms of the following eigendecomposition,
\[
\Gamma_k=
\sigma^2
V\operatorname{diag}_n
\left(\frac{\left(
1-(1-\omega s_i^2)^k
\right)^2
}{s_i^2}\right)
V^T.
\]
where the $s_i$ are singular values of $A$. Thus the eigenvalues of the covariance matrix $\Gamma_k$ %associated with the $i-$th singular vector with corresponding singular value $s_i^2$ of the matrix $B:=A^TA$ is 
are given by
$\eta_i(k)=\sigma^2
\frac{\left(1-(1-\omega s_i^2)^k\right)^2}{s_i^2}.$
% \seanc{Did some minor rewriting above. Please check.}
% \ac{Checked!}
Now for each $i$ as $k$ grows $(1-(1-\omega s_i^2)^k)\to 1$ and as such $\eta_i(k)\approx \frac{\sigma^2}{s_i^2}$ which becomes large for small $s_i^2$ and for some non-zero $\sigma^2$. Note that early stopping does offer some form of regularization so that $(1-(1-\omega s_i^2)^k)\leq 1$ for a finite stopping index $K$, nonetheless the contribution from eigenvectors corresponding to small eigenvalues $s_i^2$ of the matrix $A^TA$ start having an oversized effect on the width of the confidence intervals. Additionally, the issue is exacerbated by the fact that the global bias estimate $b_K^{ub}(\nu,A)$ may be too pessimistic for constructing practically useful confidence regions as it will take a large number of Landweber iterations to bring the global bias estimate to match with the variance term at any given noise level, $\sigma$. This motivates treating the contribution from the larger eigenvalues of $A^TA$ separately from that of the smaller eigenvalues, since the latter correspond to poorly resolved directions and can lead to substantial noise amplification. In the next section, we therefore introduce a randomized rank-$r$ approximation $\widetilde\Gamma_{K,r}$ for the principal covariance contribution and account for the omitted spectral component separately. This additional regularization is in the same spirit as the one introduced in \cite{ito2020} wherein regularization offered by using randomized truncated SVD approximation was used in conjunction with another classical regularization method, such as Tikhonov.
}

% Finally, we note that the confidence region calculated above is expressed in terms of the exact $\Gamma_K$. However computing this in large scale inverse-problems requires knowing the spectral decomposition of $A^TA$ and then iteratively calculating $\Gamma_K$ by applying Theorem \ref{thm:3.1}. This makes evaluating the exact $\Gamma_K$ computationally prohibitive. In the next subsection, we show how to leverage modern randomized numerical linear algebra techniques to efficiently approximate the covariance matrix $\Gamma_K$ and subsequently how to account for the approximation in giving confidence regions with appropriate coverage results.

\subsection{Randomized SVD for practical UQ}\label{sec:3.6}

% \seanc{I've done a bit of rewriting of the next two paragraphs.}
 
{{Our basic idea is to introduce a low rank approximation $\widetilde{B}_r \approx B := A^T A$ for the purpose of practical uncertainty quantification since computation of the full matrices $\Gamma_k$ is not practrical for large scale problems. However, we emphasize that our point-estimate for the truth comes still from the Landweber iteration \eqref{eq:LW_est}. It is only in evaluating the propagated uncertainty that we resort to randomized approximation of $B$. Hence our method is analogous to the ones proposed in \cite{ito2020}, where randomized SVD is combined with classical regularization methods, such as Tikhonov regularization, rather than being used as a replacement for regularization itself. This is also consistent with the source condition imposed in section \ref{sec:src}, i.e $x^\dagger=(A^TA)^\nu w$. If instead, we replaced $A^TA:=B$ in the Landweber iteration with an approximation, $\widetilde{B}_r$, then the corresponding Landweber iterates based on $\widetilde B_r$ would remain confined to its range, $\mathcal R(\widetilde B_r)$. Consequently, components of $x^\dagger$ lying outside the retained rank-$r$ spectral subspace could not be recovered, introducing an additional truncation bias into the point estimate. See also the discussion in \cite [section 3]{ito2020} where the authors are motivated by similar considerations when combining the randomized SVD approach with classical regularization methods.}}
\par {{Motivated by the discussion in the previous section, we now introduce the randomized low-rank approximation for practical uncertainty quantification of ill-posed inverse problems. Before we introduce our ideas, let us identify the two sources of inefficiency in how we construct the pixelwise confidence intervals that make the approach explained in section \ref{sec:3.5} theoretically valid but practically almost useless. First of all, in computing the stopping rule, we resorted to using a global bias bound which severely overestimates the true bias at each pixel. This forces the stopping rule given by \eqref{exact_stop} to be attained later as the global bias bound decreases very slowly. Now for the reason explained towards the end of section \ref{sec:3.5}, a later stopping rule inflates the width of the confidence region because the effect of smaller singular values, $s_i^2$, of the exact normal operator $B=A^TA$ starts to overpower the regularizing effect due to early stopping as stopping index becomes large. In our numerical simulations, we compute the stopping rule by replacing the overall global bias estimate by a bias estimate coming from the truncated randomized SVD and balance it against a variance term coming from the same randomized SVD approximation. If we think of randomized SVD as an additional source of regularization on top of the regularization provided by early stopping in classical Landweber, then this innovation is in the same spirit  as of \cite{ito2020}. These ideas will be rigorously explored in the subsections below. We first begin with an analysis of the approximate covariance and an approximation to the global bias via (deterministic) truncated SVD (TSVD). Subsequently, we will practically implement these approximations with the help of randomized SVD.   }}

\subsubsection{Bias and Covariance approximation for TSVD}\label{sec:3.7}

To introduce our idea, first consider the truncated rank $r-$ SVD approximation of $A^T A = B\approx B_r=V_r
\operatorname{diag}_r(s_1^2,\ldots,s_r^2)
V_r^T$ in computations, which would correspondingly yield the following approximate covariance matrix \[\Gamma_{K,r}=\sigma^2 V_r \operatorname{diag}
\left(\frac{\left(1-(1-\omega s_i^2)^K\right)^2}{
s_i^2}\right)_{i=1}^r
V_r^T\] on substituting in \eqref {errordist}. Here $V_r$ is the truncated matrix whose columns are the same as the first $r-$ columns of the matrix $V$. 
Let us consider the projection matrix, $P_r:=V_rV_r^T$. The following lemma gives us an estimate on the contribution to the bias of the part of $x^\dagger$ outside of the range of $P_r.$
\begin{lemma}\label{lem:bias_split}
 Let assumption \ref{src_assum} hold. For ease of notation, we denote the spectral norm $\rho(G):=\rho<1$ for the choice of relaxation parameter $\omega$ as in section \ref{sec:3.1}. Then, the bias term in Theorem \ref{thm:3.1} can be written as$$G^Kx^\dagger=G^KP_rx^\dagger+G^K(I-P_r)x^\dagger$$ where
 \begin{align}
    \|G^KP_rx^\dagger\|_2
\leq
M
\max_{1\leq i\leq r}
|1-\omega s_i^2|^K
s_i^{2\nu}, \,\text{ and }\,  \|G^K(I-P_r)x^\dagger\|_2
\leq
\rho^K M s_{r+1}^{2\nu} .
 \end{align}
\end{lemma}
% \seanc{$\rho$ needs to be defined. Also, is $\lambda_i = s_i^2$? We should stick with the use of $\lambda_i$ or $s_i$ only and be consistent. This also goes with the definition of $\delta_k$ above.}
% \ac{done.}
\begin{proof}
    Writing
$w=\sum_{i=1}^n w_i v_i,
$ the source condition gives
$x^\dagger=B^\nu w=\sum_{i=1}^n
s_i^{2\nu} w_i v_i.
$ Thus if we look at the projection of the bias onto the first $r-$ eigenvectors of $B$, then we have \[\|G^KP_rx^\dagger\|_2
\leq
M
\max_{1\leq i\leq r}
|1-\omega s_i^2|^K
s_i^{2\nu}.\] Finally the contribution to the bias from the tail is given by
$(I-P_r)x^\dagger
=\sum_{i=r+1}^n
s_i^{2\nu} w_i v_i,
$
and hence
\[
\|(I-P_r)x^\dagger\|_2^2
=\sum_{i=r+1}^n
s_i^{4\nu}|w_i|^2 \leq
s_{r+1}^{4\nu}
\sum_{i=r+1}^n|w_i|^2 \\
\leq
s_{r+1}^{4\nu}\|w\|_2^2
\leq
M^2s_{r+1}^{4\nu}.
\]
Moreover, since
$\|G\|_2=\rho<1$,
\[
\|G^K(I-P_r)x^\dagger\|_2
\leq
\rho^K\|(I-P_r)x^\dagger\|_2
\leq \rho^K M s^{2\nu}_{r+1}.\]
\end{proof}
The proof of Lemma \ref{lem:bias_split} tells us that for large enough $r$ and $K$, the contribution from the tail term may in fact be even smaller if the projection of the truth $x^\dagger$ onto the subspace given by $(I-P_r)$ is small, i.e. $\{w_i\}_{i=r+1}^n$ is small.
\begin{remark}\label{rem:2}
   The tail estimate in Lemma \ref{lem:bias_split} can be sharpened by retaining the actual contribution of $w$ in the omitted eigendirections. Indeed,
$(1-P_r)x^\dagger = \sum_{i=r+1}^n
s_i^{2\nu} \langle w,v_i\rangle v_i,
$ so that $
\|(I-P_r)x^\dagger\|_2
\leq
s_{r+1}^{2\nu}
\left(
\sum_{i=r+1}^n
|\langle w,v_i\rangle|^2
\right)^{1/2}.
$ Hence,
$$
|G^K(I-P_r)x^\dagger|2
\leq
\rho^K s_{r+1}^{2\nu}
\left(
\sum_{i=r+1}^n
|\langle w,v_i\rangle|^2
\right)^{1/2}.
$$
Thus, the tail contribution will be small not only when $s_{r+1}$ is small, but also when $w$ has little contribution in the eigendirections $v_{r+1},\ldots,v_n$. The bound in Lemma \ref{lem:bias_split} follows by replacing this tail norm by the global bound $\|w\|_2\leq M$. While verifying that such a condition would require knowledge of $x^\dagger$, in our numerical simulations we indeed find this to be the case. This motivates the stopping rule \eqref{eq:3.36} %\seanc{I fixed this reference, and I think this is correct but please check.}
below in which we ignore the contribution from the tail in the bias to suggest a stopping rule that is better suited for practical UQ than the one given above, \eqref{exact_stop}.
\end{remark}

Evaluating the truncated SVD (TSVD) still requires knowledge of the (partial) SVD of $A^TA$ which for a large matrix is inefficient to compute. To propose a stopping rule on the basis of TSVD is thus computationally infeasible for large scale inverse problems. Thus, in this project we propose to use a truncated randomized SVD (rSVD) approximation in the manner described in the subsection below.

% $B\approx \widetilde B_r$ such that $\widetilde B_r$ is also a symmetric positive semi-definite matrix. The corresponding approximation to covariance matrix will then be denoted, $\widetilde{\Gamma}_{K,r}$. The proposed stopping rule is obtained by balancing the pixelwise approximate variance term against the approximate bias obtained from the rSVD. Working with approximate covariance would also require the confidence regions to be suitably inflated in order to achieve proper coverage as we show below.
\subsubsection{Stopping rule and confidence regions from rSVD}\label{sec:3.6.2}
% In the previous subsection we considered a deterministic rank $r$- truncated SVD approximation of the normal operator, $B:=A^TA$. In particular, we rigorously analysed the contribution to the propagated covariance and regularization bias from the first $r$ eigenvalues of the the normal operator and in fact, Remark \ref{rem:2} shows that contribution to the bias from the emitted eigendirections may be small if a large enough rank $r$ is chosen and simultaneously the component of $w$ in the omitted eigen directions $\{v_{i}\}_{i=r+1}^n$is small. We recall that from the source condition we assume the existence of a $w, \|w\|<M$, such that $x^\dagger = B^{\nu} w$. \seanc{I'd advocate cutting everything in this section above here, since I think it is repeating things that have already been said.}
As computing the deterministic truncated SVD is computationally infeasible for large-scale inverse problems, we
replace the truncated SVD approximation by a randomized rank-$r$ approximation and use it to construct the
stopping rule and the corresponding uncertainty quantification procedure. We start by constructing a randomized rank $r-$ approximation to the matrix $B\approx \widetilde{B}_r=\widetilde{V_r}\widetilde{\Lambda}_r\widetilde{V}_r^T$ using \cite[Algorithm 5.3]{halko2011} as described in Lemma \ref{lem:1}. {{We note that constructing such a randomized rank $r$-approximation requires a choice of an `oversampling parameter' $p\geq 4$ alongwith the target rank $r$. For details, we refer the reader to \cite{halko2011}.}} Subsequently, the approximate covariance matrix is computed using an analog of the formula from Theorem \ref{thm:3.1} in the following way. Let us denote $\widetilde{G}_r=(I-\omega \widetilde{B}_r).$ For any iteration $k$, we have the approximated propagated covariance given by:
\begin{align}
  \widetilde\Gamma_{k,r}
&=\sigma^2\omega^2
\left(\sum_{j=1}^{k-1} \widetilde{G}_r^j \right)\widetilde{B}_r\left(\sum_{m=1}^{k-1}\widetilde{G}_r^m\right) \nonumber\\
&=
\sigma^2
\widetilde V_r
\operatorname{diag}
\left(
\frac{
\left(
1-(1-\omega\widetilde{s}_i^2)^k
\right)^2
}{
\widetilde{s}_i^2
}
\right)_{i=1}^r
\widetilde V_r^T.
\end{align}
Its diagonal entries are given explicitly by
\begin{align}
\label{eq:rsvd_cov_diag}
\widetilde{\gamma}_{k,r,j}:=(\widetilde\Gamma_{k,r})_{jj}
=\sum_{i=1}^r
(\widetilde V_r)_{ji}^2
\sigma^2\frac{\left(1-(1-\omega\widetilde{s}_i^2)^k
\right)^2
}{\widetilde{s}_i^2
}.
\end{align}
Similarly, motivated by remark \ref{rem:2}, we define the randomized bias approximation at any iteration $k$ as:
\begin{align}
\label{eq:rsvd_bias}
\widetilde b_{k,r}
:= M\max_{1\leq i\leq r}
\left|1-\omega\widetilde{s}_i^2\right|^k\widetilde{s}_i^{2\nu}.
\end{align}
We note that the quantity $\widetilde{b}_{k,r}$ is only an approximation of the true bias at iteration $k$ of the Landweber iteration and we will separately account for the part of the bias not captured by the randomized SVD approximation below. We next account for the error introduced by replacing the exact propagated covariance $\Gamma_k$ by the randomized approximation $\widetilde{\Gamma}_{k,r}$. Rather than estimating the diagonal entries of the entire covariance matrix $\Gamma_k$ using Monte Carlo sampling, we estimate only the diagonal entries of the residual matrix
\[
D_{k,r}
:=
\Gamma_k-\widetilde{\Gamma}_{k,r}.
\]
Since both $\Gamma_k$ and $\widetilde{\Gamma}_{k,r}$ are symmetric, the matrix $D_{k,r}$ is symmetric. Thus, the Monte Carlo diagonal approximation results of \cite{hallman2023} can be applied directly to $D_{k,r}$. Let $w_1,\ldots,w_N\in\mathbb R^n$ be independent Rademacher vectors, and define
\[
d_{k,r,j}
:=(D_{k,r})_{jj}
=(\Gamma_k-\widetilde{\Gamma}_{k,r})_{jj}.
\]
We estimate this quantity using
\begin{align}\label{residual}\widehat d_{k,r,j}
=\frac{1}{N}
\sum_{\ell=1}^N
(D_{k,r}w_\ell)_j
(w_\ell)_j.
\end{align}
Similar to before, from \cite[Corollary 5.4]{hallman2023} for prescribed
$0<\epsilon<1$ and $0<\delta<1$, and for sufficiently large $N$,
\[
\mathbb P
\left(\left|\widehat d_{k,r,j}
-d_{k,r,j}
\right|\leq\epsilon
\left| d_{k,r,j}\right|
\right)\geq
1-\delta.
\]
%Furthermore, since $(\Gamma_k)_{jj}=(\widetilde{\Gamma}_{k,r})_{jj}+ (D_{k,r})_{jj}$. 
Thus, on the event {{$\left|\widehat d_{k,r,j}
-d_{k,r,j}
\right|\leq\epsilon \left| d_{k,r,j} \right|$}} we have {{$\left | d_{k,r,j}\right|\leq \frac{\left|\widehat{d}_{k,r,j}\right| }{1-\epsilon} $. Furthermore we note that the diagonal entries of both the covariance matrix and its randomized approximation are non-negative on account of being positive semi-definite (PSD) matrices. In this case, since $(\Gamma_k)_{jj}=(\widetilde{\Gamma}_{k,r})_{jj}+ (D_{k,r})_{jj}$, we get
\[(\Gamma_k)_{jj}\leq(\widetilde{\Gamma}_{k,r})_{jj}+ \frac{\left|\widehat{d}_{k,r,j}\right| }{1-\epsilon}.\]
% \seanc{Somthing is wrong with the reasoning above. We may need to slightly change this.}
We note that the advantage of estimating just the residual $(D_{k,r})_{jj}$ over estimating the entirety of $(\Gamma_{k})_{jj}$ using a Monte Carlo technique is that as shown in \cite[Corollary 5.1]{hallman2023}, the componentwise diagonal-estimation error depends on the off-diagonal contribution in the corresponding row of the matrix whose diagonal is being estimated. Hence, when the randomized rank $r-$ SVD approximation captures the dominant covariance structure well, estimating the
diagonal of $D_{k,r}$ can be substantially more accurate than that of estimating $\Gamma_k$ from same number of Rademacher samples. Then similar to what was done  in section \ref{sec:3.7}, we therefore define the variance quantity used in the stopping rule as
\begin{align}\label{eq:3.35}
\widehat{v}_{k,r,N}(\sigma,A)
:=\max_{1\le j\le n}
\sqrt{\left|\left(\widetilde{\Gamma}_{k,r}\right)\right|_{jj}
+\frac{\left|\widehat{d}_{k,r,j}\right|}{1-\epsilon}
}.
\end{align}
The stopping index is then defined by
\begin{align}\label{eq:3.36}
K_{r,N}(\sigma)
:=\inf\left\{
k\ge 1:\widetilde{b}_{k,r}
\le \widehat{v}_{k,r,N}(\sigma,A)
\right\}.
\end{align} 
% \seanc{Is $c_0$ necessary? If so, how is it chosen?}
% \ac{removed it, the `balancing' only requires them to be of the 'same' order. But in our simulations we had $c_0=1$}
Thus similar to the exact SVD setting of section \ref{sec:3.7}, the stopping rule given by \eqref{eq:3.36} balances the approximation to the deterministic regularization bias against an approximation to the propagated pixelwise uncertainty given by the diagonal terms of the covariance matrix. We also note that for a given noise level $\sigma$ and for the choices of retained rank $r$, oversampling parameter, $p$ 
% \seanc{$p$ has not been mentioned so far.}
% \ac{mentioned above in blue at the start of section 3.6.2 now.}
and number of Rademacher random vectors, $N$, this quantity is computed offline independently of the observed data. 

We will now explain our construction of the pixelwise confidence intervals when working with the stopping rule \eqref{eq:3.36}. First, in line with our observation in Remark \ref{rem:2}, under the source condition $x^\dagger =B^\nu w$, the contribution of the omitted eigendirection to the regularization bias is small, if the components of $w$ are small along those omitted eigendirections. We formalize this observation in the following assumption.
\begin{assum}\label{assum2}
 Let
$\widetilde P_r := \widetilde V_r \widetilde V_r^{T}
$
denote the orthogonal projection onto the randomized rank-$r$ subspace {{as defined in Lemma \ref{lem:1} in the appendix}} , and let
\[
\mathcal{E}_{r,p}
:=\left\{
\|B-\widetilde B_r\|_2
\leq
\beta_{r,p}
\right\}
\]
denote the high-probability event associated with the randomized rank-$r$
approximation of $B=A^TA$, see Lemma \ref{lem:1} below.
We assume that, on the event $\mathcal{E}_{r,p}$, the contribution to the regularization bias from the directions omitted by the randomized approximation
is no larger than the contribution from the retained directions. That is,
for the iteration indices $k$ under consideration,
\[\left\|
G^k (I-\widetilde P_r)x^\dagger
\right\|_2
\leq
\eta_{r,k}\left\|
G^k \widetilde P_r x^\dagger
\right\|_2
\] for some $0\leq \eta_{r,k}\ll 1$ and $\eta_{r,k}\to 0$ as $r$ increases.
We assume that {{on the event $\mathcal E_{r,p}$}} the retained contribution is bounded by the
randomized bias approximation
$\left\|
G^k \widetilde P_r x^\dagger
\right\|_2
\leq
\widetilde b_{k,r},
$
where
$
\widetilde b_{k,r}
:=M \max_{1\leq i\leq r}\left|1-\omega \widetilde s_i^2\right|^k
\widetilde s_i^{2\nu}.
$
Consequently, on $\mathcal{E}_{r,p}$,
\[\|G^k x^\dagger\|_2
\leq
(1+\eta_{r,k})\,\widetilde b_{k,r}.
\]   
\end{assum}
% \seanc{$\widetilde{P}_r$ has not been defined. I presume $\widetilde{P}_r = \widetilde{V}_r \widetilde{V}_r^T$? Also, isn't the final assumption that the randomised bias is bounded automatically satisfied?}
% \ac{I think we should keep this as an assumption. Because the true bias is being computed using the true $G$ and not $\tilde{G}$. So the $\tilde{b}$ term is coming from the approximation but the LHS of that has essentially $x^\dagger=G^k(\widetilde{P}_r x^\dagger+(I-\widetilde{P}_r)x^\dagger)$. For our earlier argument with true $P$ and true $G$ (in truncated SVD subsection), the retained bias is upper bounded by $b_{k,r}$ necessarily, but I am not sure about whether that exact upper bound also holds when we have true $G$ but randomized $P_r$ projector. In particular what I want to say is that $G^k \widetilde{P}_r x^\dagger \neq \widetilde{G}^k \widetilde{P}_r x^\dagger\leq \widetilde{b}_{k,r}$. I think we can make a high probability statement, but for now I leave that as an assumption. Also explicitly stated where $P_r$ is coming from.}

Let, $K_{r,N}(\sigma):=\hat{K}$ be the stopping index from \eqref{eq:3.36}.
Let us define $\widehat{\gamma}^{\mathrm{up}}_{\hat K,r,j}
:=\widetilde{\gamma}_{\hat K,r,j}
+\frac{|\widehat d_{\hat K,r,j}|}{1-\varepsilon}.$ Then we will construct the practical pixelwise confidence intervals  by,
\begin{align}\label{eq:hatC}\widehat C_{\hat K,r,j}(Y)
:=\left\{
(x_\sigma^{\hat K})_j
\pm
\left[
z_{1-\alpha/2}
\sqrt{
\widehat{\gamma}^{\,\mathrm{up}}_{\hat K,r,j}
}
+(1+\eta_{r,k})\widetilde b_{\hat K,r}
\right]
\right\}.\end{align} We have the following coverage result for pixelwise confidence intervals constructed as such.
\begin{theorem}\label{pixelwiseUQ}
 Let the assumptions of Lemma \ref{thm:3.2}, Lemma \ref{lem:1} and Assumptions \ref{src_assum} and \ref{assum2} hold. Let $0\leq\alpha,\delta,\epsilon\leq 1$. We assume that the noise in the data is independent of the randomized SVD construction and the choice of random Rademacher vectors to estimate the diagonal of the residual matrix as above. Then for every pixel we have,
 \[\mathbb{P}_{x^\dagger}(x^\dagger_j\in \widehat C_{K,r,j}(Y))\geq (1-\alpha)(1-3e^{-p})(1-\delta)\]
\end{theorem}
\begin{remark}
Note that Theorem \ref{pixelwiseUQ} says that there exist choices of $r,\,p$ and number of independent Rademacher probles, $N$, such that the lower bound in the coverage guarantee can be made as close to the desired coverage level given by $\alpha$. In practice, however, the constructed confidence intervals will almost always have coverage greater than the desired coverage, as is also evident from the simulations in Section \ref{sec:4}.
\end{remark}
\begin{proof}
    Recall that at the stopping index, $K_{r,N}(\sigma):=\hat{K}$ the reconstruction error can be decomposed as \[ x^\dagger - x_\sigma^{\hat{K}} = G^{\hat K} x^\dagger + Z_{\hat K}, \qquad Z_{\hat K} \sim N(0,\Gamma_{\hat K}). \] Therefore, for any fixed pixel $j\in \{1,\dots,n\}$, as in the proof of Theorem \ref{thm:3.3}, we have \[\mathbb{P}\left( |(Z_{\hat K})_j| \leq z_{1-\alpha/2}\sqrt{(\Gamma_{\hat K})_{jj}} \right) = 1-\alpha.
    \]
    Now on the randomized SVD event $\mathcal{E}_{r,p}$ defined in Lemma \ref{lem:1}, we have from Assumption \ref{assum2}, \[ |(G^Kx^\dagger)_j| \leq \|G^Kx^\dagger\|_2 \leq (1+\eta_{r,K})\widetilde b_{K,r}. \] Also for sufficiently large number of Rademacher probes, $N$, the error in the MC estimation of the covariance residual is given as
% \seanc{Maybe it should be ``estimation of the covariance residual" or something like that} 
\[
\mathbb P
\left(
\left|
\widehat d_{\hat{K},r,j}
-d_{\hat{K},r,j}
\right|
\leq
\epsilon
\left|
d_{\hat{K},r,j}
\right|
\right)
\geq
1-\delta.
\] Since, $(\Gamma_{\hat K})_{jj}=\gamma_{\hat K,j}
=\widetilde{\gamma}_{\hat K,r,j}+d_{\hat K,r,j},$ we get:
\[\gamma_{\hat K,j}
\leq\widetilde{\gamma}_{\hat K,r,j}+\frac{|d_{\hat K,r,j}|}{1-\epsilon}=\widehat{\gamma}^{\,\mathrm{up}}_{\hat K,r,j}.
\] Finally on the set $E_{r,p}$ as defined in Lemma \ref{lem:1} (see below \eqref{eq:App2}), we have, $\mathbb{P}_{\mathrm{rSVD}}(E_{r,p})\geq 1-3e^{-p}.$ Thus under the assumption that the randomized SVD construction and the estimation of the diagonal of residual matrix using Rademacher probes is independent of the noise in the observed data, we have  \[\mathbb{P}_{x^\dagger}(x^\dagger_j\in \widehat C_{\hat K,r,j}(Y))\geq (1-\alpha)(1-3e^{-p})(1-\delta)\]

\end{proof}

\section{Numerical Simulation Results}\label{sec:4}
We now illustrate the proposed probabilistic Landweber framework developed in section \ref{sec:3.6} on a two-dimensional computerized tomography problem. In particular, we will show the practical construction of the randomized covariance approximation, the bias-variance stopping rule, and the resulting pixelwise confidence intervals. We also demonstrate frequentist coverage of the proposed confidence intervals by repeated noisy-data experiments.
\subsection{Simulation set-up}
To begin our numerical experiments, we consider the standard Shepp-Logan phantom of size $128\times 128$ on a Cartesian grid. We note that the corresponding parameter vector $x^\dagger \in \mathbb{R}^{16384}$. The forward data is generated by applying MATLAB's inbuilt function `radon' to the phantom at $180$ equally spaced angles, which corresponds to the so-called full-data case. The resulting noise-free sinogram is vectorized and Gaussian noise is added to it. Thus the simulated forward data can be modeled as, \[Y = Ax^\dagger+\epsilon,
\quad\epsilon\sim N(0,\sigma^2 I),
\]
where we will loosely refer to $\sigma$ as the noise-level. Next we implement the Landweber method given by \eqref{eq:3.1}. The numerical implementation of the method uses the Radon transform and unfiltered back-projection in MATLAB to apply $A$ and $A^T$, respectively, so that the matrix $A^T A$ does not need to be explicitly constructed. The Landweber iteration is initialized with $x^{(0)}=0$. The relaxation parameter $\omega$ is chosen after estimating the largest eigenvalue, $\hat{s_i}_{\max}^2$, of the matrix, $A^TA$. Note that the largest eigenvalue of the normal operator $A^T A$ is again estimated by a matrix-free power iteration, using only repeated applications of the forward and adjoint operators. We set the maximum number of Landweber iterations in our method to $K_{\max}=1000.$  For the source condition as mentioned in section \ref{sec:src}, we use $\nu=1/2.$ To justify this choice, note that the Shepp-Logan phantom is in $H^{\alpha}(\mathbb{R}^2)$ for $\alpha < 1/2$, while the normal operator for the Radon transform $R^* R$ is a pseudodifferential operator of order $-1$. Thus, we expect the range of $(R^* R)^{1/2}$ on $L^2(\mathbb{R}^2)$ to be $H^{1/2}(\mathbb{R}^2)$. Because of this, we expect that the Shepp-Logan phantom satisfies the source condition for $\nu< 1/2$, although this is a heuristic rather than rigorous proof. Practically, we may as well use $\nu = 1/2$ rather than $\nu = 1/2-
\epsilon$ for small $\epsilon$.
We also mention, \cite[Proposition 8.10]{Engl_book}, where the choice of $\nu=1/2$ is shown to correspond to a range type condition .  %{\textcolor{blue} {This is because for Landweber iteration applied to such linear ill-posed problems, typically one assumes $0<\nu<1$, however certain discrepancy based stopping rules require $0<\nu\leq 1/2$, see e.g., \cite[section 1]{gerth_19}.}}

{Now we outline the construction of the practical pixelwise confidence region using Theorem \ref
{pixelwiseUQ}. To that end, after fixing a target rank, $r$ and an oversampling parameter $p=10$, randomized approximation to $B\approx \widetilde{B}_r$ is done according to \cite[Algorithm 5.3]{halko2011}. Note that in order to implement the stopping rule to find $K_{r,N}(\sigma):= \hat K$, we estimate $\widetilde{b}_{k,r}$ for every itertaion $k\leq \hat{K}$ from \eqref{eq:rsvd_bias}. In our numerical simulations, we used the fixed value of $M=6$. Next we compute the residual $\widehat d_{k,r,j}$ as per \eqref{residual}  using a Monte-Carlo estimator. We also fix $c_0=1$ and thereby find stopping index according to rule given by \eqref{eq:3.36}. We note that since the exact number of Rademacher probes $N$ needed to determine the diagonal elements of the residual are not exactly known, we choose $N$ large enough, so that $1-\epsilon \approx 1$. In simulations, we fixed $N=100$. We also assume that rank $r$ is chosen so that $1+\eta_{r,K}\approx 1$ as implied by Assumption \ref{assum2}. Once the stopping index $\hat K$ is found for a fixed noise level $\sigma$ and for chosen target rank $r$ with oversampling parameter $p=10$,  the Landweber method runs for $\hat K$ many iterations.  Subsequently, pixelwise confidence intervals are built according to \eqref{pixelwiseUQ}. Finally, we also asses the frequentist coverage of pixelwise confidence intervals. For these experiments, the true image $x^\dagger$ is fixed and independent realizations of the observational noise are generated according to
\[\varepsilon^{(t)}
\sim N(0,\sigma^2I),
\quad t=1,\ldots,N_{\mathrm{trial}}.
\]
Thus, for each trial, we have the observed data given by
\[
y^{(t)}=Ax^\dagger+\varepsilon^{(t)},
\]
and the Landweber reconstruction is computed using the same stopping index $\hat{K}$ determined in the offline stage.  For each pixel $j$, the empirical frequentist coverage is defined by
\[\widehat{\mathrm{Cov}}_j=
\frac{1}{N_{\mathrm{trial}}}
\sum_{t=1}^{N_{\mathrm{trial}}}
\mathbf 1
\left\{
x^\dagger_j
\in
 \widehat C_{\hat K,r,j}^{(t)}
\right\}.\]
}
\subsection{Simulation results and interpretation}
The results from the numerical simulations carried at different noise levels and for different randomized rank$-r$ approximations of the normal operator are tabulated in Table \ref{tab:1}. In figure \ref{fig:nine_figs}, we have given example reconstructions with randomized SVD rank chosen to be $r=900$ and we do reconstructions across three noise levels given by, $\sigma=0.1, 0.2$ and $0.3$. For each noise level, we plot the corresponding reconstruction, confidence intervals built as in  \eqref{eq:3.36} for pixels that fall on the horizontal line through the center of the phantom and a corresponding empirical coverage map for all the pixels across $100$ independent trials. In figure \ref{fig:six_figs}, we plot the reconstructions at a fixed noise level, $\sigma=0.2$ for two different choices of rank, $r=500$ and $r=1300$ respectively. The plotted results agree with the interpretation given below for the trends observed in Table \ref{tab:1}.
%%%%%%%%%%%%%%%%%%%%%%%%%%%%%%%%%%%%%%%%%%% Fig 1 %%%%%%%%%%%%%%%%%%
\begin{figure}[htbp]
    \centering

    % Row 1
    \begin{subfigure}[b]{0.30\textwidth}
        \centering
        \includegraphics[width=\textwidth]{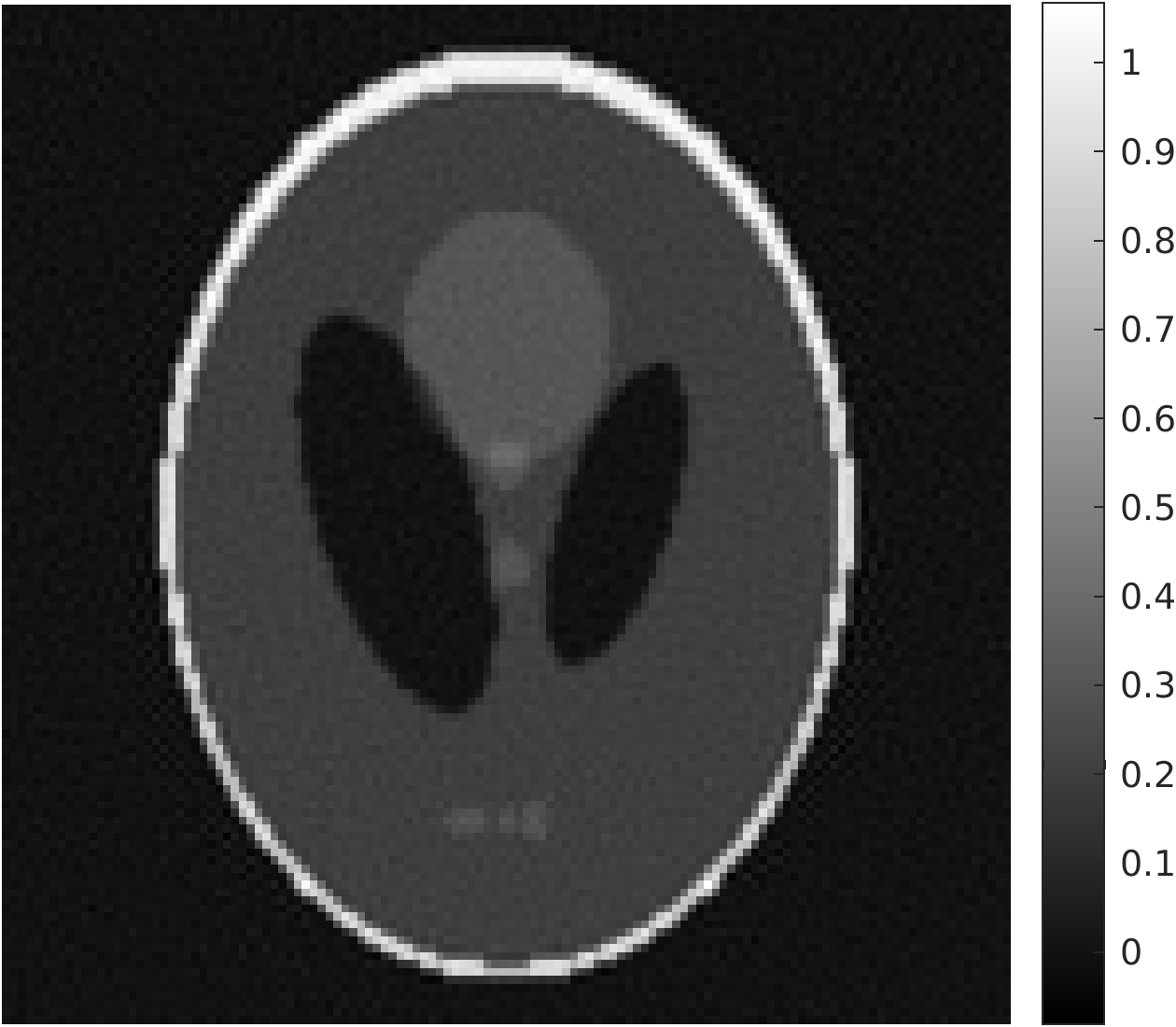}
        \caption{reconstruction at $\sigma=0.1$}
        \label{fig:1}
    \end{subfigure}
    \hfill
    \begin{subfigure}[b]{0.30\textwidth}
        \centering
        \includegraphics[width=\textwidth]{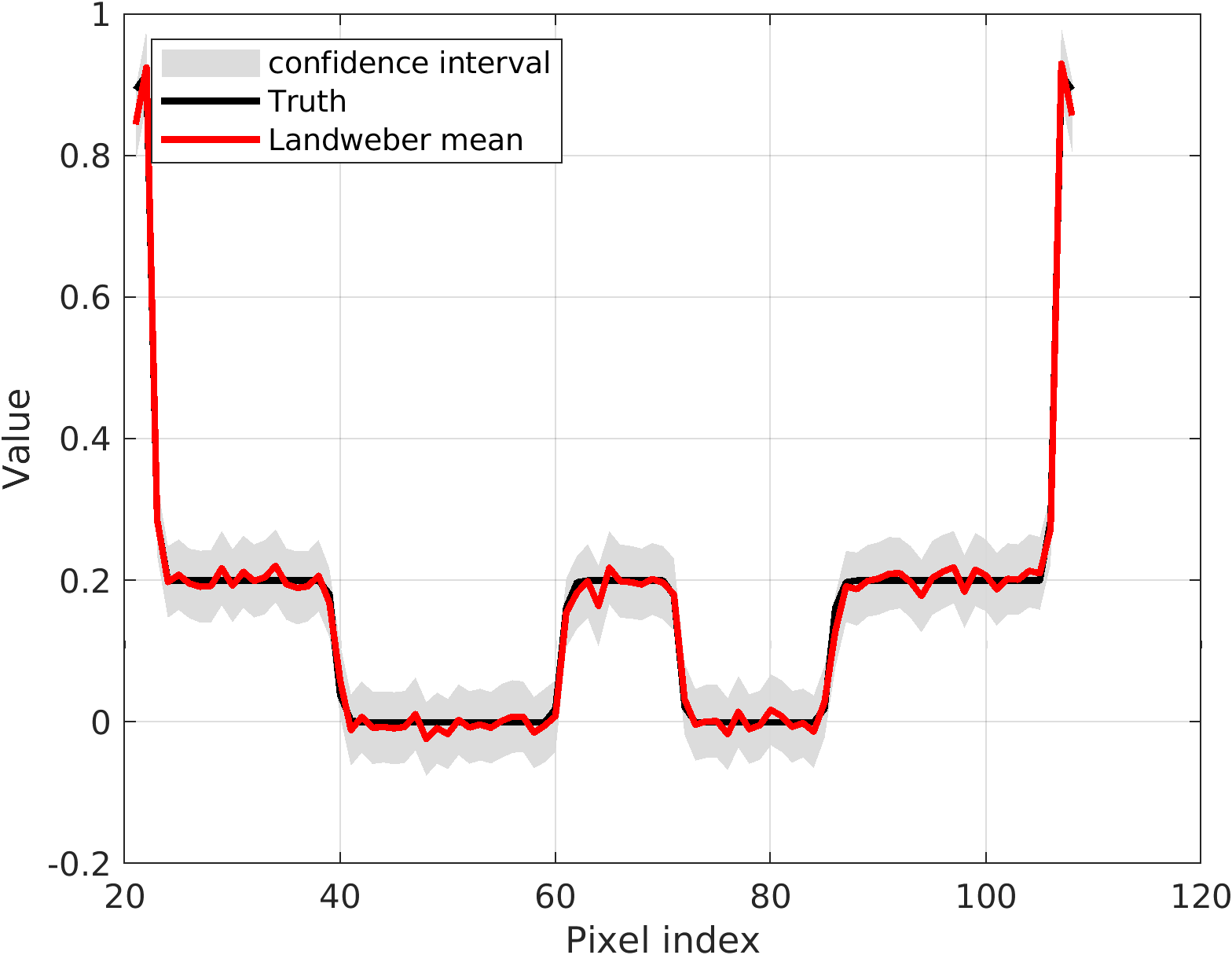}
        \caption{Confidence intervals along the horizontal center line, $\sigma=0.1$}
        \label{fig:2}
    \end{subfigure}
    \hfill
    \begin{subfigure}[b]{0.30\textwidth}
        \centering
        \includegraphics[width=\textwidth]{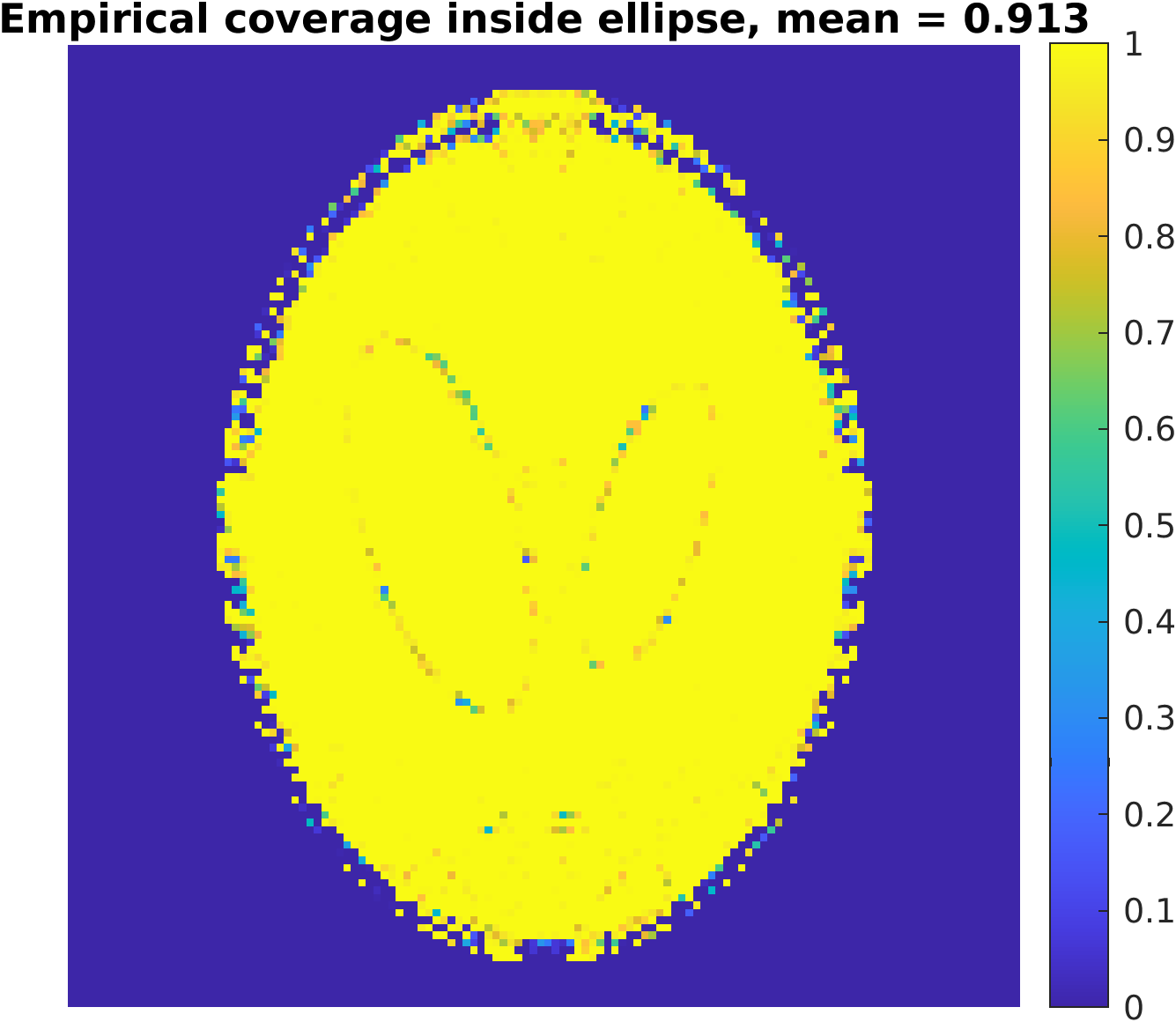}
        \caption{Empirical Coverage map across 100 trials, $\sigma=0.1$}
        \label{fig:3}
    \end{subfigure}

    \vspace{0.4cm}

    % Row 2
    \begin{subfigure}[b]{0.30\textwidth}
        \centering
        \includegraphics[width=\textwidth]{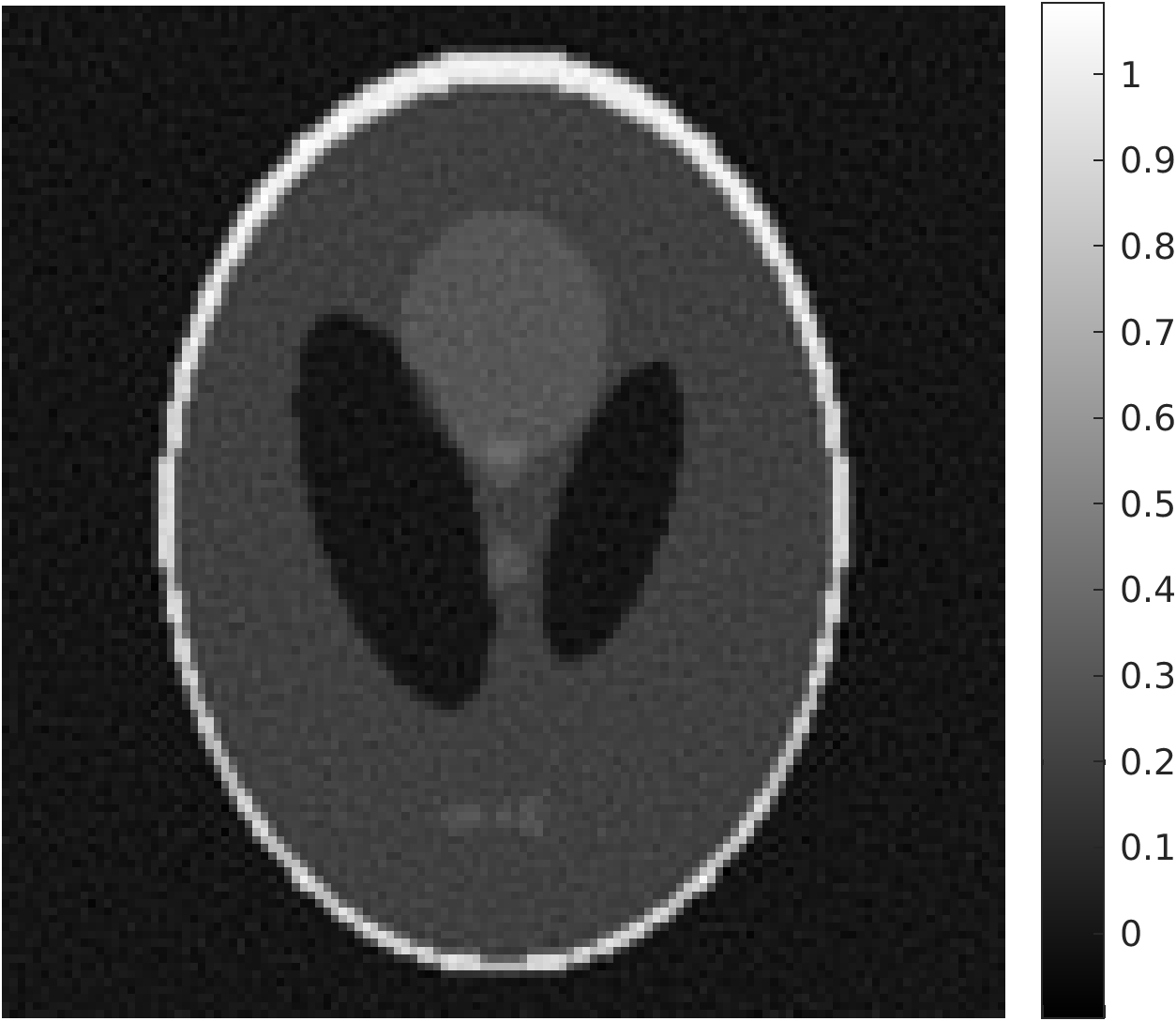}
        \caption{reconstruction at $\sigma=0.2$}
        \label{fig:4}
    \end{subfigure}
    \hfill
    \begin{subfigure}[b]{0.30\textwidth}
        \centering
        \includegraphics[width=\textwidth]{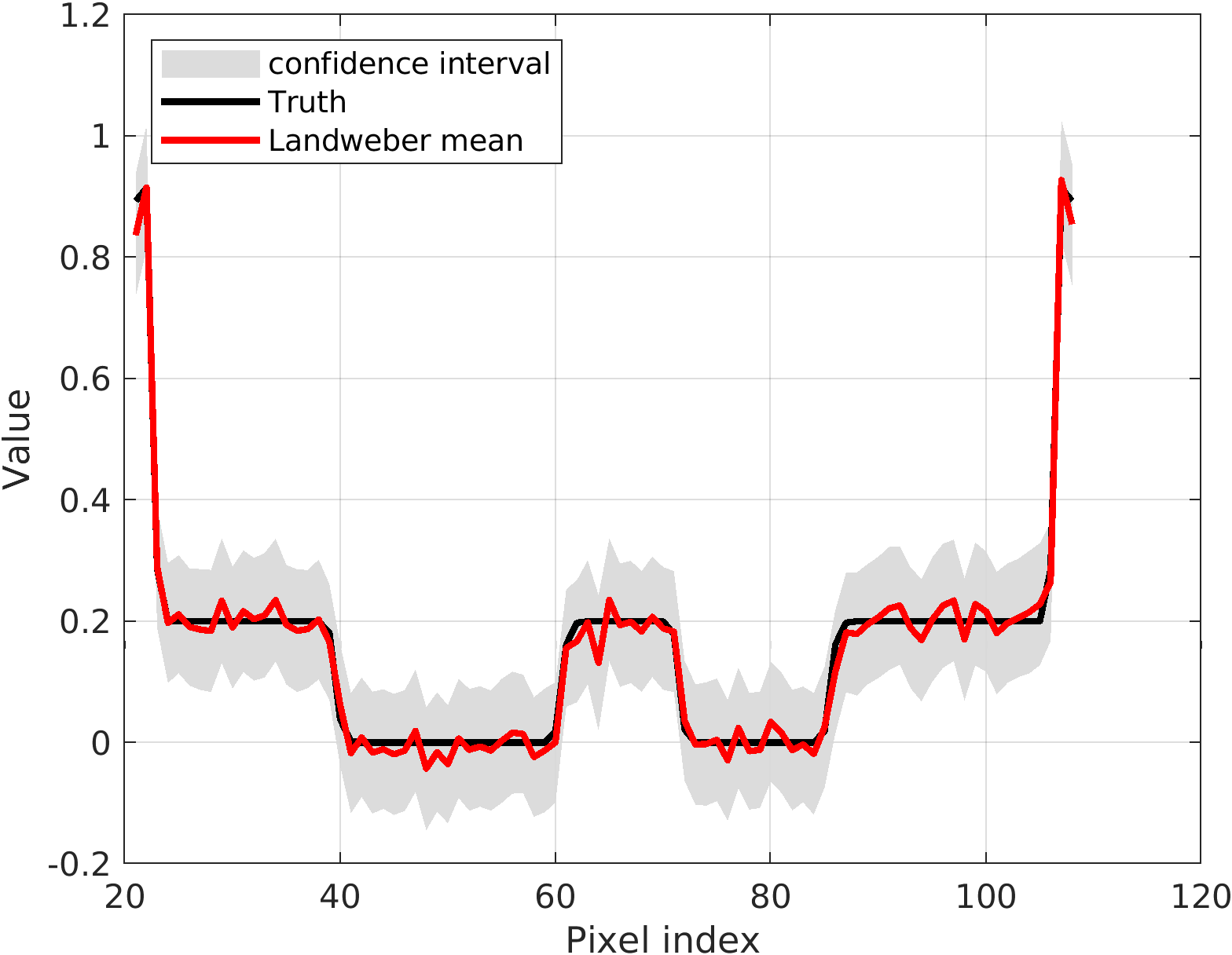}
        \caption{Confidence intervals along the horizontal center line, $\sigma=0.2$}
        \label{fig:5}
    \end{subfigure}
    \hfill
    \begin{subfigure}[b]{0.30\textwidth}
        \centering
        \includegraphics[width=\textwidth]{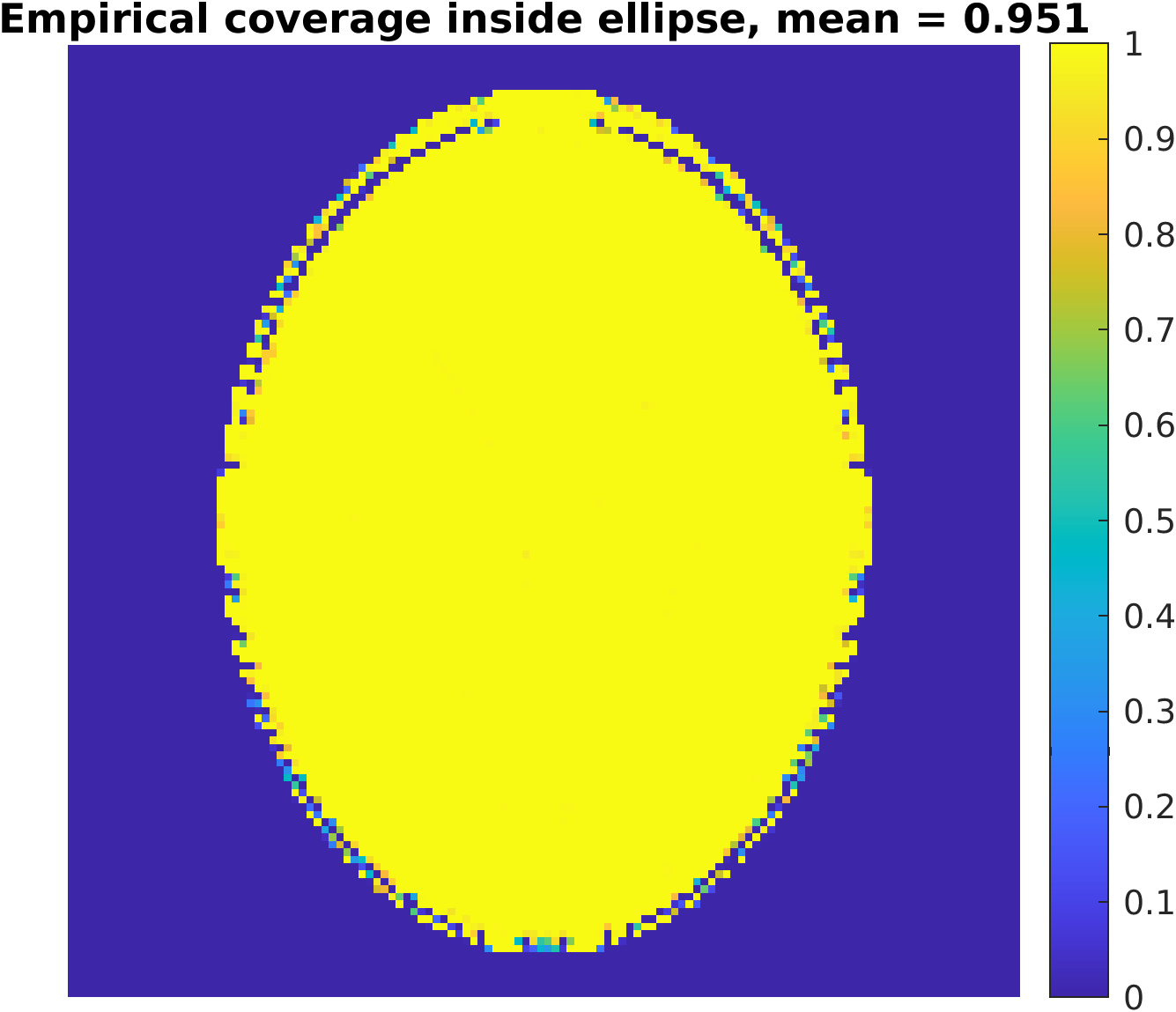}
        \caption{Empirical Coverage map across 100 trials, $\sigma=0.2$}
        \label{fig:6}
    \end{subfigure}

    \vspace{0.4cm}

    % Row 3
    \begin{subfigure}[b]{0.30\textwidth}
        \centering
        \includegraphics[width=\textwidth]{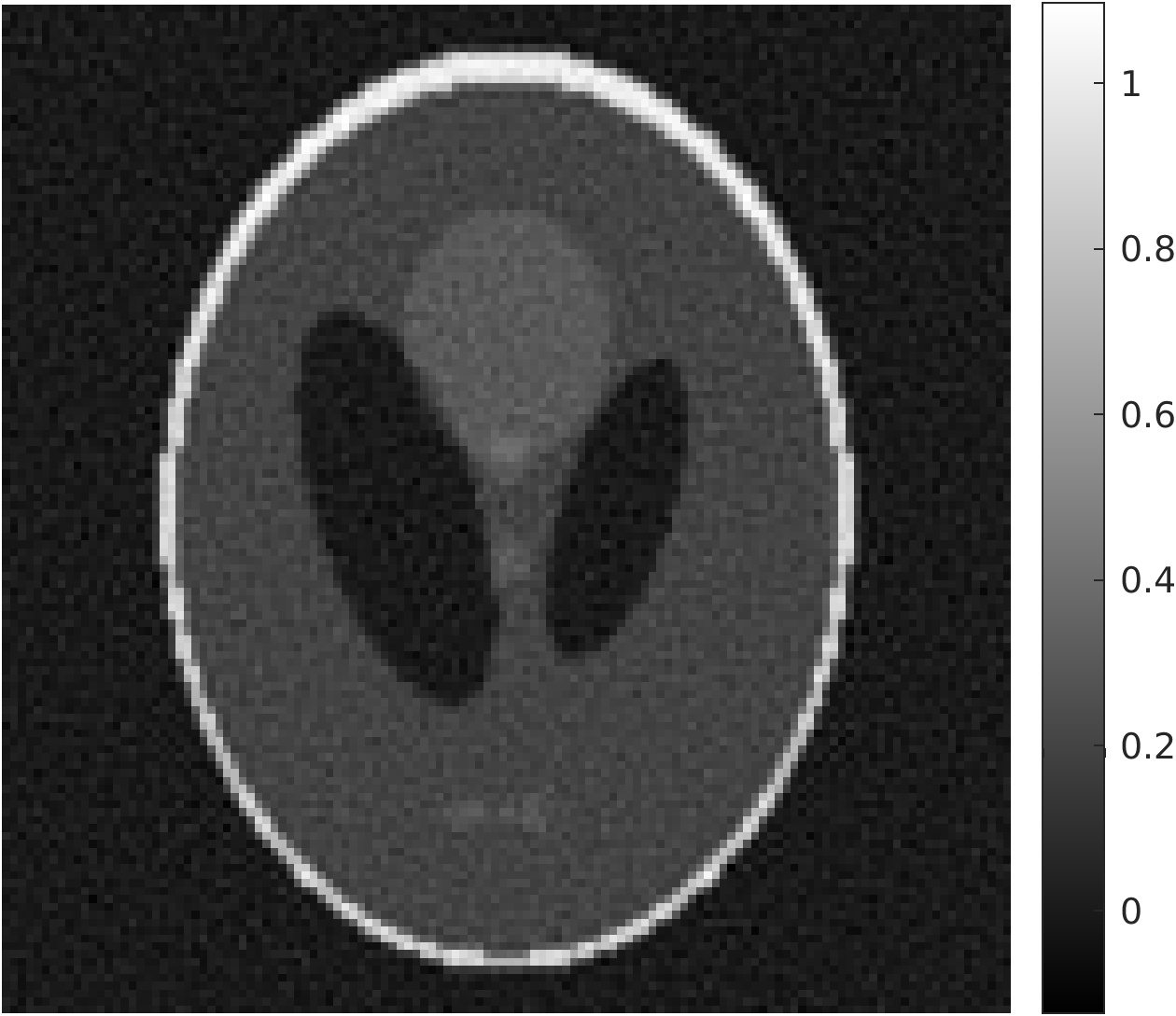}
        \caption{reconstruction at $\sigma=0.3$}
        \label{fig:7}
    \end{subfigure}
    \hfill
    \begin{subfigure}[b]{0.30\textwidth}
        \centering
        \includegraphics[width=\textwidth]{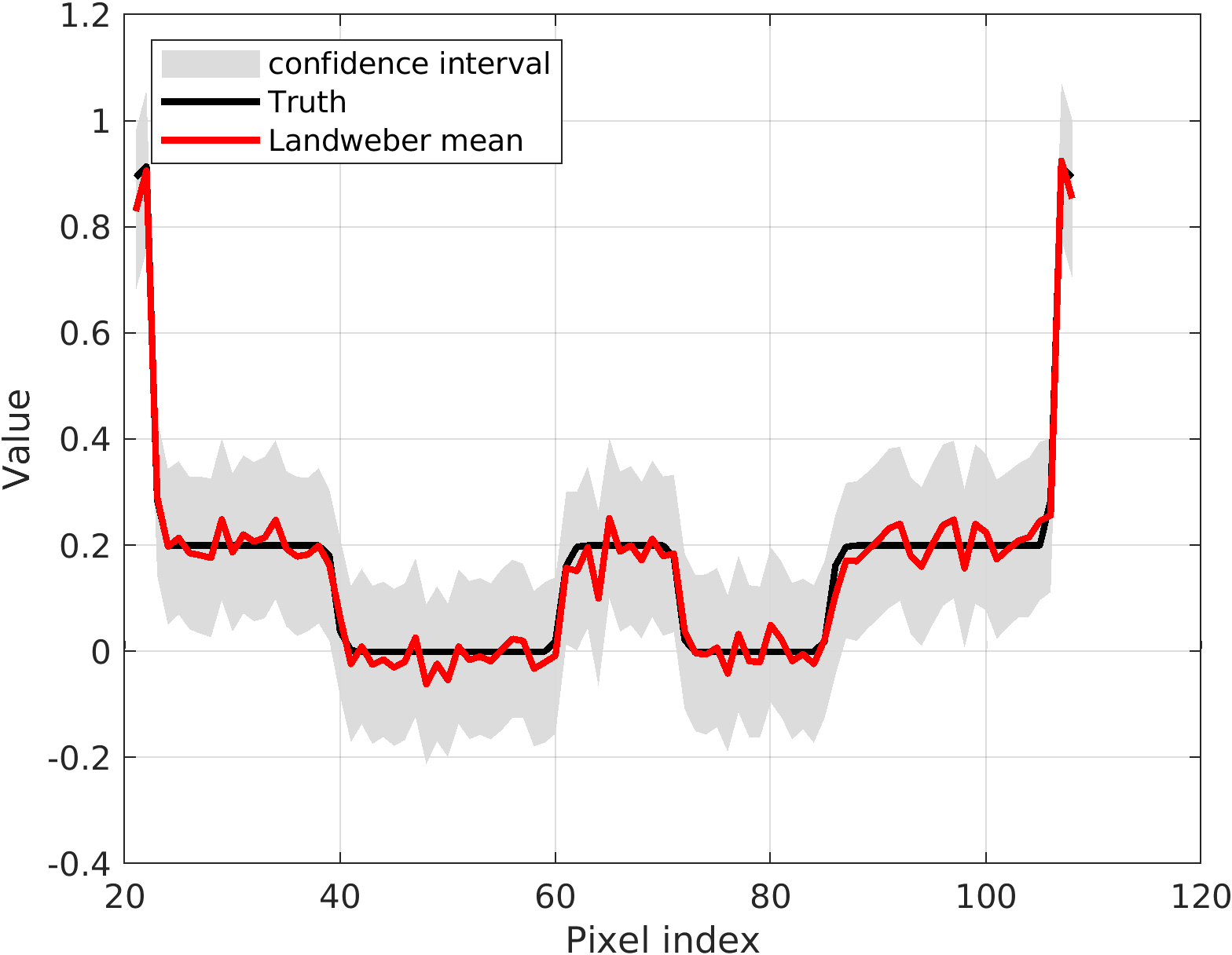}
        \caption{Confidence intervals along the horizontal center line, $\sigma=0.3$}
        \label{fig:8}
    \end{subfigure}
    \hfill
    \begin{subfigure}[b]{0.30\textwidth}
        \centering
        \includegraphics[width=\textwidth]{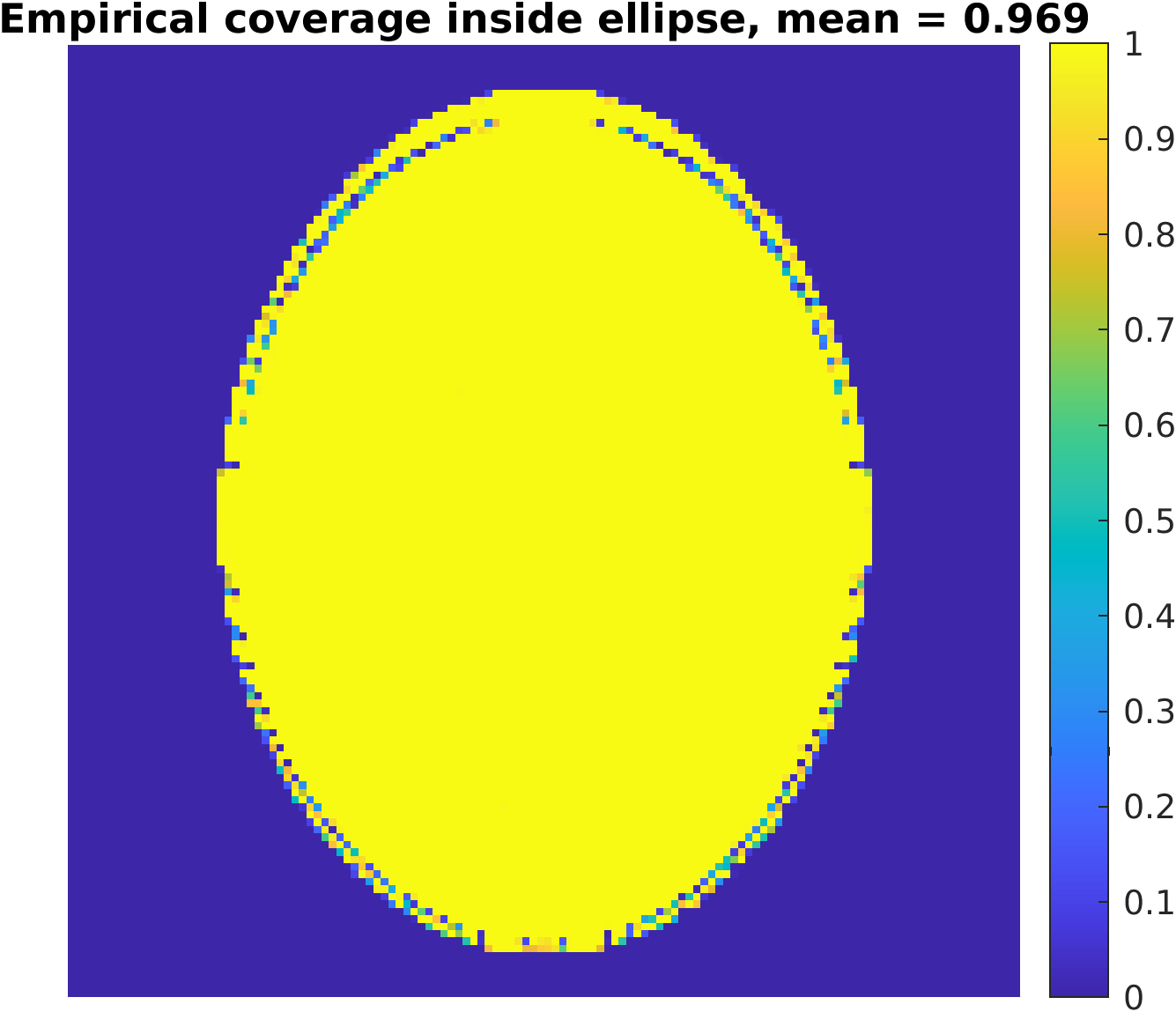}
        \caption{Empirical Coverage map across 100 trials, $\sigma=0.3$}
        \label{fig:9}
    \end{subfigure}

    \caption{Reconstructions and covergae maps for retained rank = $900$ for $\sigma=0.1,0.2,$ and $0.3$ respectively}
    \label{fig:nine_figs}
\end{figure}
%%%%%%%%%%%%%%%%%%%%%%%%%%%%%%%%%%%%%%%%%%%%% Fig 2 %%%%%%%%%%%%%%%%
\begin{figure}[htbp]
    \centering

    % Row 1
    \begin{subfigure}[b]{0.30\textwidth}
        \centering
        \includegraphics[width=\textwidth]
        {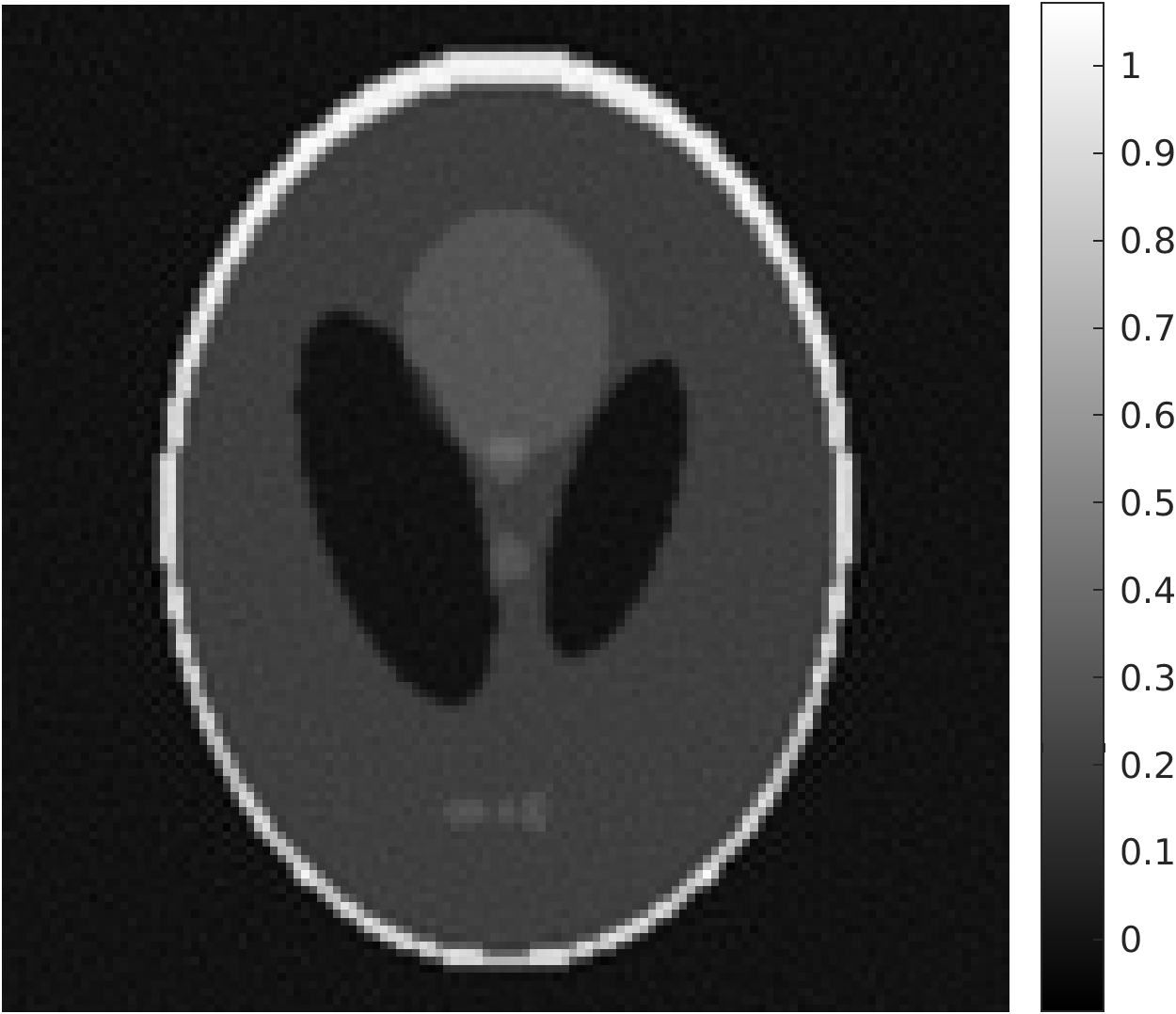}
        \caption{Reconstruction at rank $500$ for $\sigma=0.1$.}
        \label{fig:rank500_reconstruction}
    \end{subfigure}
    \hfill
    \begin{subfigure}[b]{0.30\textwidth}
        \centering
        \includegraphics[width=\textwidth]
        {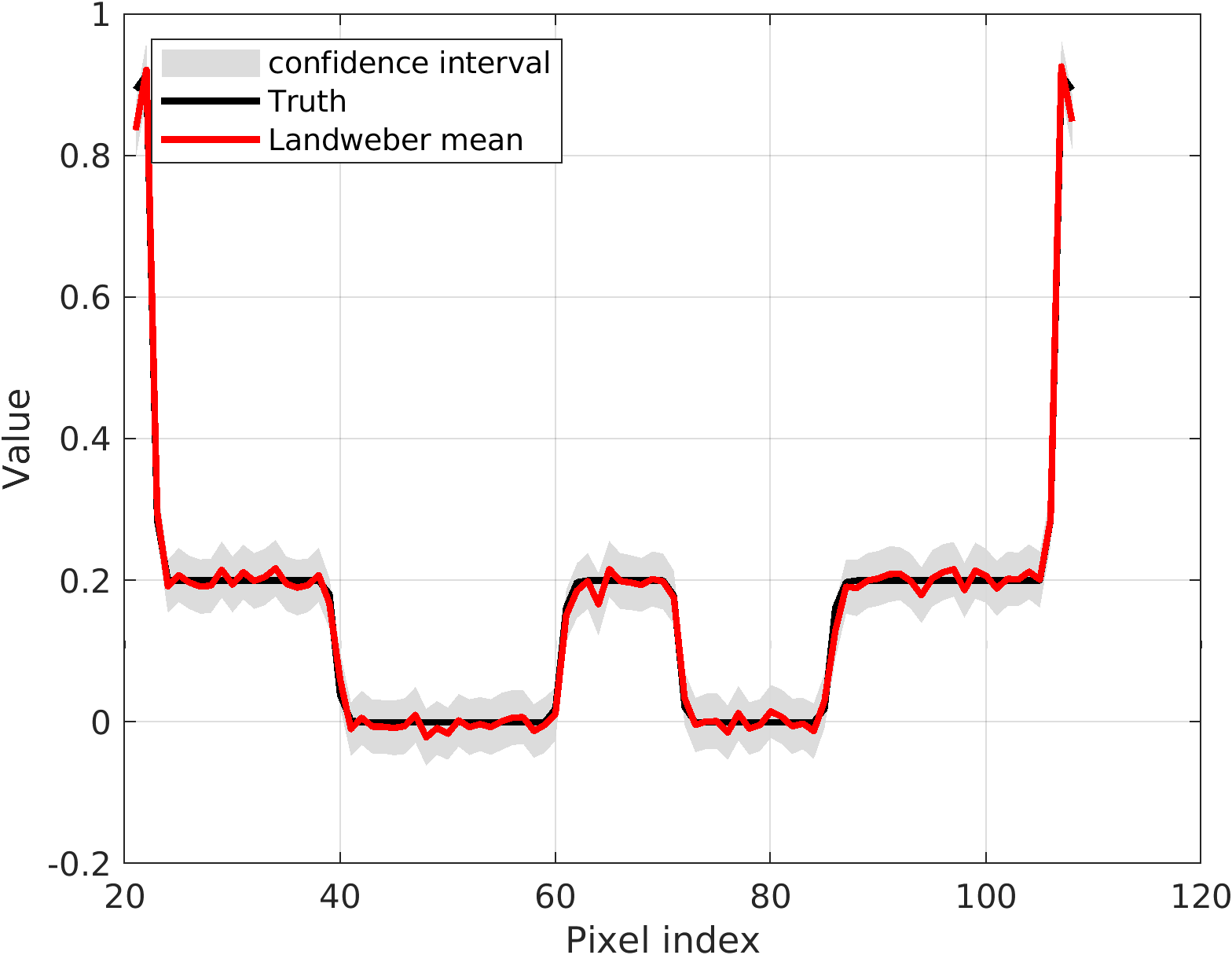}
        \caption{Confidence intervals along the horizontal center line for rank $500$ and $\sigma=0.1$.}
        \label{fig:rank500_ci}
    \end{subfigure}
    \hfill
    \begin{subfigure}[b]{0.30\textwidth}
        \centering
        \includegraphics[width=\textwidth]
        {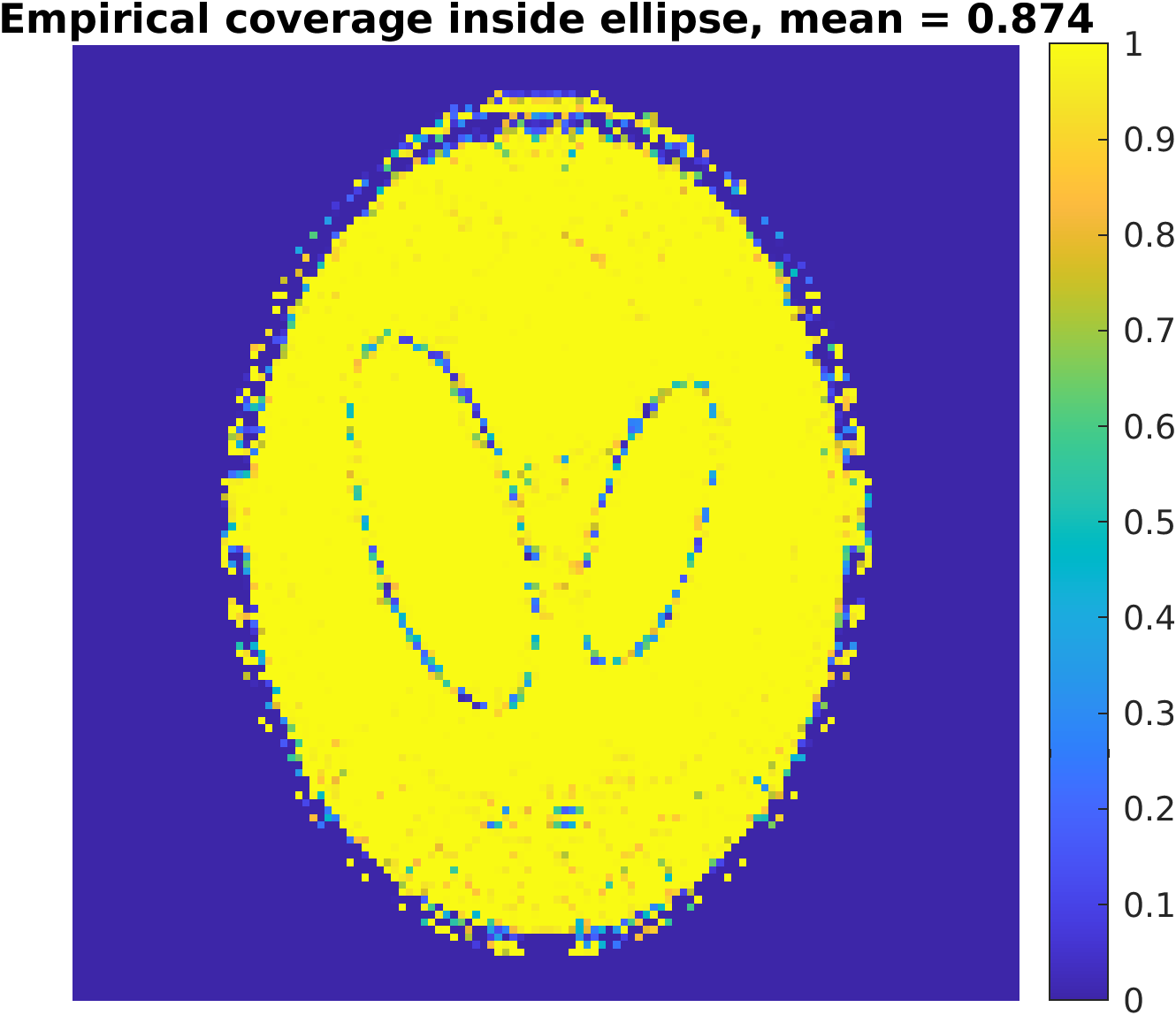}
        \caption{Empirical coverage map across 100 trials for rank $500$ and $\sigma=0.1$.}
        \label{fig:rank500_coverage}
    \end{subfigure}

    \vspace{0.4cm}

    % Row 2
    \begin{subfigure}[b]{0.30\textwidth}
        \centering
        \includegraphics[width=\textwidth]
        {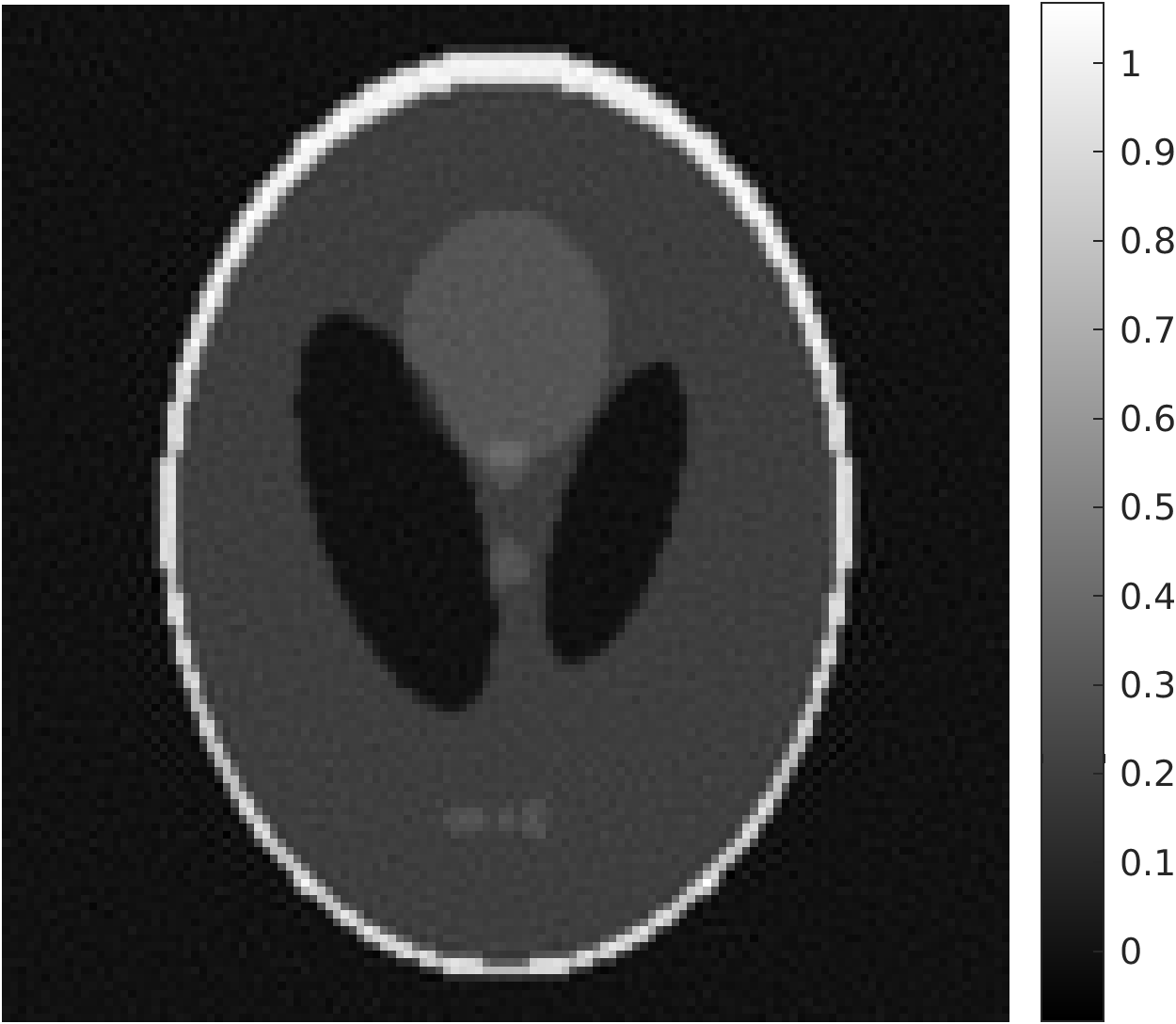}
        \caption{Reconstruction at rank $1300$ for $\sigma=0.1$.}
        \label{fig:rank1300_reconstruction}
    \end{subfigure}
    \hfill
    \begin{subfigure}[b]{0.30\textwidth}
        \centering
        \includegraphics[width=\textwidth]
        {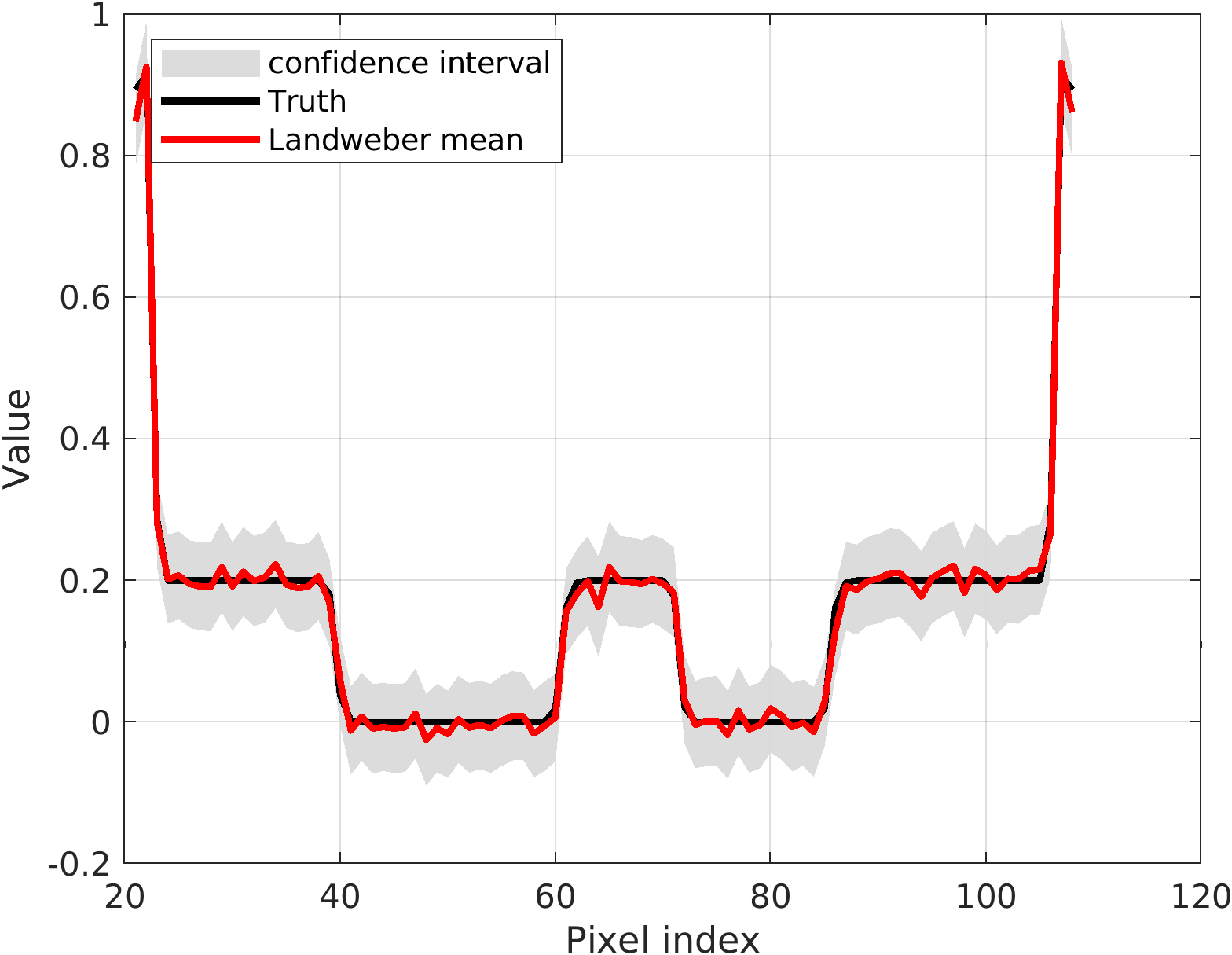}
        \caption{Confidence intervals along the horizontal center line for rank $1300$ and $\sigma=0.1$.}
        \label{fig:rank1300_ci}
    \end{subfigure}
    \hfill
    \begin{subfigure}[b]{0.30\textwidth}
        \centering
        \includegraphics[width=\textwidth]
        {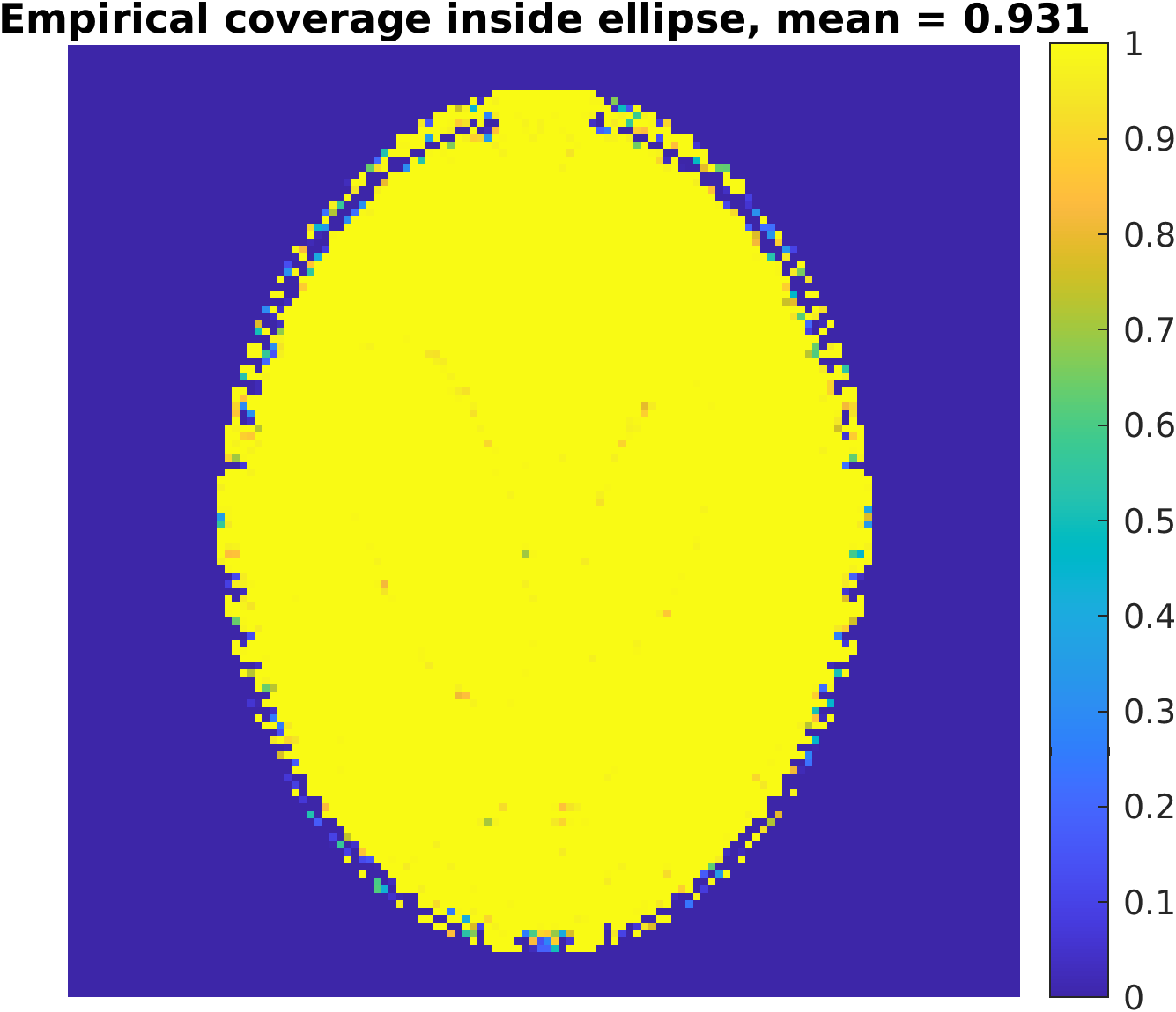}
        \caption{Empirical coverage map across 100 trials for rank $1300$ and $\sigma=0.1$.}
        \label{fig:rank1300_coverage}
    \end{subfigure}

    \caption{Reconstructions, confidence intervals, and empirical coverage maps
    for retained ranks $500$ and $1300$ with $\sigma=0.1$.}
    \label{fig:six_figs}
\end{figure}

%%%%%%%%%%%%%%%%%%%%%%%%%%%%%%%%%%%%%%%%%% Table %%%%%%%%%%%%%%%%%%%
\begin{table}[htbp]
\centering
\caption{Empirical coverage results for different randomized SVD ranks and noise levels. 
The desired nominal coverage level is 0.95. Interior coverage is computed inside the phantom, 
while the one-dimensional coverage is computed along the horizontal center (row 64)  of the phantom. The stopping index is obtained by balancing bias and variance as in \eqref{eq:3.36}.}
\label{tab:1}
\begin{tabular}{c c c c c}
\hline
\textbf{Rank} & $\boldsymbol{\sigma}$ & \textbf{Stopping Index $\hat K$} 
& \textbf{Mean Interior Coverage} 
& \textbf{Mean Coverage (Row 64)} \\
\hline

\multirow{3}{*}{500}
& 0.1 & 304 & 0.8740 & 0.9255 \\
& 0.2 & 273 & 0.9316 & 0.9653 \\
& 0.3 & 254 & 0.9507 & 0.9871 \\
\hline

\multirow{3}{*}{700}
& 0.1 & 350 & 0.8968 & 0.9381 \\
& 0.2 & 312 & 0.9432 & 0.9798 \\
& 0.3 & 290 & 0.9599 & 0.9972 \\
\hline

\multirow{3}{*}{900}
& 0.1 & 393 & 0.9134 & 0.9462 \\
& 0.2 & 349 & 0.9512 & 0.9918 \\
& 0.3 & 323 & 0.9693 & 0.9995 \\
\hline

\multirow{3}{*}{1100}
& 0.1 & 432 & 0.9236 & 0.9530 \\
& 0.2 & 381 & 0.9579 & 0.9977 \\
& 0.3 & 352 & 0.9790 & 0.9999 \\
\hline

 \multirow{3}{*}{1300}
 & 0.1 & 467 & 0.9311 & 0.9625 \\
 & 0.2 & 411 & 0.9644 & 0.9995 \\
 & 0.3 & 379 & 0.9872 & 0.9999 \\
\hline

\end{tabular}
\end{table}
Indeed, Table \ref{tab:1} shows {three} clear trends. First, for a fixed noise level $\sigma$, increasing the randomized SVD rank generally increases empirical coverage. This is particularly noticeable at the lowest noise level, $\sigma=0.1$, where the mean interior coverage increases from $0.8740$ at rank $500$ to $0.9311$ at rank $1300$. This suggests that, when the observational noise is small, truncation of the singular spectrum becomes an important source of error: lower-rank approximations omit components that are sufficiently important relative to the noise level, leading to undercoverage.

Second, for a fixed rank, empirical coverage increases with the noise level. For example, at rank 700, the mean interior coverage rises from $0.8968$ for $\sigma=0.1$ to $0.9432$ for $\sigma=0.2$ and $0.9599$ for $\sigma=0.3$. At higher noise levels, however, increasing the retained rank also introduces progressively smaller singular values into the covariance approximation. As discussed in section \ref{sec:3.6}, these directions can make disproportionately large contributions to the propagated variance, with the contribution behaving approximately like $\sigma^2/s_i^2$ for sufficiently large iteration index. Consequently, once the dominant spectral components have been adequately captured, increasing the rank further may primarily inflate the confidence regions, leading to increasingly conservative coverage.

{{Third, for a fixed rank, as noise increases the stopping rule causes the iterations to stop earlier. Besides, for the same noise level increasing the rank allows the iterations to go on further. This behavior is consistent with the interpretation of both early stopping and spectral truncation as complementary forms of regularization. When the noise level is larger, stronger regularization through earlier stopping is required to limit the effect of noise. On the other hand, a smaller retained rank already provides stronger spectral regularization and hence causes an earlier stopping point. As the retained rank increases and the spectral truncation becomes less severe, the stopping rule compensates by allowing more Landweber iterations.}}

Overall, the results indicate a trade-off between the randomized SVD rank and the nominal coverage at a given observational noise level. If the rank is too small, important spectral directions are omitted and the resulting confidence intervals may undercover, particularly at low noise levels. Conversely, at larger noise levels, retaining too many small-singular-value directions can inflate the propagated uncertainty and produce overcoverage. Thus, the appropriate randomized SVD rank should depend on the observational noise level: lower noise requires a sufficiently large rank to control truncation-induced undercoverage, whereas at higher noise levels a smaller effective rank may already provide coverage close to the nominal $0.95$ level. This suggests the need to develop an adaptive randomized SVD approach to suitably choose a rank for a given noise level. We leave this to be addressed in a future work.
\subsection*{Acknowledgment} The authors began this research as a part of a focused research group at BIRS, Canada in August 2025 and completed it while being a part of Collaborate @ ICERM workshop in USA in August 2026. We thank both BIRS and ICERM for their hospitality and for providing pleasant, productive working conditions. S Holman was supported by the Engineering and Physical Sciences Research Council (EPSRC) through grant EP/V007742/1. AI tools have been used for improving and refactoring the MATLAB codes used in this project. 
\section{Appendix }\label{sec:app}
\subsection{Covariance matrix from SVD}
{The lemma below computes the covariance matrix $\widetilde\Gamma_{K,r}$ resulting from a randomized SVD approximation to the matrix $A^TA:=B$ in Theorem \ref{thm:3.1}.} 
\begin{lemma}\label{lem:1}
 Let $B=A^TA,\ 
G=I-\omega B$ and let $\widetilde B_r:=\widetilde{V_r}\widetilde{\Lambda}_r\widetilde{V}_r^T$ be the randomized rank $r-$ SVD approx to $B$ obtained by using \cite[Algorithm 5.3]{halko2011} with $r-$ as the traget rank and $p$ as an oversampling parameter. We also define $\widetilde G_r:=I-\omega\widetilde B_r$. Let the step size $\omega$ be chosen such that, $\|G\|\leq\rho<1$ and $\|\widetilde G_r\|\le1.$ For notational convenience we denote $S(G):=\sum_{k=0}^{K-1}G^k 
$ and similarly for $S(\widetilde{G}_r)$. We note that the matrices, $G,\, \widetilde{G}_r$ commute respectively with $B,$ and $\widetilde{B}_r$. Then we can write the covariance matrix from Theorem \ref{thm:3.1} and its approximation resulting from  truncated randomized SVD approximation to $B$ respectively as, \begin{align}
    \Gamma_K&=
\sigma^2\omega^2 S(G)\,B\, S(G)=\sigma^2\omega^2B(S(G))^2,\\
\widetilde\Gamma_{K,r}
&=\sigma^2\omega^2
S(\widetilde G_r)\,\widetilde B_r\,S(\widetilde G_r)=\sigma^2\omega^2\widetilde{B}_r(S(\widetilde G_r))^2.
\end{align}
Then $\mathbb{P}_{\mathrm{rSVD}}(\lVert\Gamma_K-\widetilde{\Gamma}_{K,r}\rVert\leq \Delta_{K,r,p})\geq(1-3e^{-p})$ for any $p\geq 4.$ Here, $\Delta_{K,r,p}$ is defined below.
\end{lemma}

\begin{proof}
% Note that, $\lVert\Gamma_K -\widetilde{\Gamma}_{K,r}\rVert_2\leq \lVert \Gamma_K-\Gamma_{K,r}\rVert_2+\lVert \Gamma_{K,r}-\widetilde{\Gamma}_{K,r}\rVert_2.$ We will now approximate each of these separately. The two terms signify the error due to (deterministic) best rank $r-$ approximation and the error due to using a randomized SVD approximation in lieu of the best rank $r$ approximation, respectively. First of all, note that for the deterministic  error,
% \begin{align}\label{det_err}
% \|\Gamma_K-\Gamma_{K,r}\|_2=\sigma^2
% \max_{i>r}
% \frac{
% \left(
% 1-(1-\omega s_i^2)^K
% \right)^2
% }{
% s_i^2
% }
% \leq
% \sigma^2\omega^2K^2s_{r+1}^2.\end{align}
% Effectively, this estimate tells us that to maintain the error below some chosen tolerance level $\varepsilon$, as the number of iterations grow, we should allow the retained rank to grow as well so that $K^2s^2_{r+1}\lesssim \varepsilon$. 

We estimate the error in approximating the covariance matrix by the randomized algorithm given in \cite[Algorithm 5.3]{halko2011}. To that end, consider first
\begin{align}
\Gamma_K&=\sigma^2\omega^2B(S(G))^2=\sigma^2 \omega(I-G)(S(G))^2\nonumber\\
    &=\sigma^2\omega \left(\sum_{k=0}^{K-1}G^k -\sum_{k=0}^{K-1}G^{k+1}\right)\sum_{k=0}^{K-1}G^k \nonumber\\
    &=\sigma^2\omega(I-G^K)\sum_{k=0}^{K-1}G^k =\sigma^2\omega \left(\sum_{k=0}^{K-1}G^k -\sum_{k=K}^{2K-1}G^k \right)
\end{align}
Similarly, it is easy to verify that $
   \widetilde\Gamma_{K,r}=\sigma^2\omega \left(\sum_{k=0}^{K-1}\widetilde G_r^k -\sum_{k=K}^{2K-1}\widetilde G_r^k \right).
  $
Thus, we get,
\begin{align*}
\Gamma_{K}-\widetilde\Gamma_{K,r}
=\sigma^2\omega
\Bigg[\sum_{k=0}^{K-1}
\left(
G^k-\widetilde G_r^k
\right)-
\sum_{k=K}^{2K-1}
\left(
G^k-\widetilde G_r^k
\right)
\Bigg].
\end{align*}

%  Consider,
% \begin{align}
% \widetilde\Sigma_{K,r}-\Sigma_{K,r} &=
% \sigma^2\omega^2\left[
% (S(\widetilde G_r)-S(G_r))
% \widetilde B_r
% S(\widetilde G_r)
% +S(G_r)(\widetilde B_r-B_r)
% S(\widetilde G_r)
% +S(G_r)B_r
% (S(\widetilde G_r)-S(G_r))\right].
% \end{align} We will now try to estimate each term separately. First of all, on taking norms,
% \begin{align}\label{eq:10.3}
% \|\widetilde\Sigma_{K,r}-\Sigma_{K,r}\|
% \le&
% \sigma^2\omega^2\left[
% \|S(\widetilde G_r)-S(G_r)\|
% \|\widetilde B_r\|
% \|S(\widetilde G_r)\|
% +\|S(G_r)\|
% \|\widetilde B_r-B_r\|
% \|S(\widetilde G_r)\|
% +\|S(G_r)\|
% \|B_r\|
% \|S(\widetilde G_r)-S(G_r)\|\right].
% \end{align}
% We first bound $\Vert S(G_r)\rVert$ and $\lVert S(\widetilde G_r)\rVert$. From our assumption $\|G_r\|\leq 1$ and $\|\widetilde G_r\|\leq 1$. Consider first, \begin{align}\label{eq:10.4}
% \|S(G_r)\|
% =\left\|
% \sum_{m=0}^{K-1}G_r^m
% \right\|
% \leq
% \sum_{m=0}^{K-1}\|G_r\|^m
% \leq
% K:=H_K.\end{align} Exactly the same calculation also gives us, $\|S(\widetilde{G}_r)\|\leq H_K$. 
% Next, we bound the norms of $\lVert S(G_r)-S(\widetilde G_r)\rVert$. 
\noindent Using the identity, 
$\widetilde G_r^m-G^m
=\sum_{j=0}^{m-1}
\widetilde G_r^{\,j}
(\widetilde G_r-G)
G^{\,m-1-j},
$
 and the fact that \(
\|G\|_2=\rho<1,\|\widetilde G_r\|_2\le1
\),  we get
\begin{align*}\|\widetilde G_r^m-G^m\|_2
\le \sum_{j=0}^{m-1}\rho^{m-1-j}\,
\|\widetilde G_r-G\|_2=\frac{1-\rho^m}{1-\rho}\omega\lVert \widetilde B_r- B\rVert_2.
\end{align*}
It follows that
\begin{align}\label{rand_err}
\|\Gamma_{K}-\widetilde\Gamma_{K,r}\|_2
&\le \sigma^2\omega
\sum_{k=1}^{2K-1}
\|G^k-\widetilde G_r^k\|_2
\le
\sigma^2\omega^2\left(
\sum_{k=1}^{2K-1}\frac{1-\rho^k}{1-\rho}
\right)
\|B-\widetilde B_r\|_2
\nonumber\\&= \sigma^2\omega^2 H_{K,\rho}
\|B-\widetilde B_r\|_2 \text{ where } H_{K,\rho}=\frac{2K-1}{1-\rho}-\frac{\rho(1-\rho^{2K-1})}{(1-\rho)^2}.
\end{align}

% Therefore,
% \begin{align}\label{eq:10.5}
% \|S(\widetilde G_r)-S(G_r)\|
% &=
% \left\|
% \sum_{m=0}^{K-1}
% (\widetilde G_r^m-G_r^m)
% \right\|
% \leq
% \sum_{m=1}^{K-1}
% m
% \|\widetilde G_r-G_r\|
% \nonumber\\
% &=
% C_K
% \|\widetilde G_r-G_r\|=C_K\omega \lVert B_r-\widetilde B_r\rVert
% \end{align}
% where
% $C_K=
% \sum_{m=1}^{K-1}
% m=K(K-1)/2$
%  and $K$ is the iteration index.

%\seanc{Not sure if it's useful, but you can simplify $C_K$ using the geometric identity since it is the derivative with respect to $\rho$ of the normal geometric sum. I think that in particular you can show 
%\[
%C_K \leq \frac{1}{(1-\rho)^2}
%\]
%for all $K$.}
\noindent It thus remains for us to estimate $\lVert B-\widetilde{B}_r\rVert_2$. In order to do this, let us unpack how \cite[Algorithm 5.3]{halko2011} is used to find the randomized SVD approximation to $B:=A^TA$. First of all let $Q$ denote the randomized range finder from \cite[Algorithm 4.1]{halko2011} (see also remark \ref{rem2} below) such that, \begin{align}\label{eq:App1}
    \lVert B- QQ^*B\rVert_2 =\lVert (I-P) B\rVert_2\leq \epsilon
\end{align} for some chosen tolerance level $\epsilon$. In practice, this corresponds to choosing a target rank $r$ and oversampling parameter $p\geq 4$ in the approximation so that \eqref{eq:App1} holds with a high probability. In particular, we have the following result from \cite[Corollary 10.9]{halko2011},
\begin{align}\label{eq:App1.2}
  \mathbb{P}_{\mathrm{rSVD}}\left(  \|(I-P)B\|_2\leq\left(
1+16\sqrt{1+\frac{r}{p+1}}
\right)s_{r+1}^2+\frac{8\sqrt{r+p}}{p+1}
\left(\sum_{j>r}s_j^4\right)^{1/2}\right)\geq 1-3e^{-p}
\end{align}
For notational convenience, let us denote $\alpha_{r,p}:=\left(1+16\sqrt{1+\frac{r}{p+1}}
\right)s_{r+1}^2+\frac{8\sqrt{r+p}}{p+1}
\left(\sum_{j>r}s_j^4\right)^{1/2}$. Now as per \cite[Algorithm 5.3]{halko2011},  for the symmetric positive semi-definite (PSD) matrix, $B$, we first form another PSD matrix, $C=Q^*BQ$, evaluate the eigen-decomposition of $C= V\Lambda V^*$ and then subsequently, $B\approx U\Lambda U^*$ where $U=QV$. Effectively then, we have the matrix $B$ approximated as follows:
\[B\approx QV \Lambda V^*Q^*=QCQ^*=QQ^*BQQ^*=PBP\]
\noindent Now $\lVert B-PBP\rVert_2\leq \lVert (I-P) B\rVert_2+\lVert PB(I-P)\rVert_2\leq 2\lVert (I-P) B\rVert_2 $. Note also that the matrix $Q\in \mathbb{R}^{n\times( r+p)}.$ Thus to get the final randomized rank $r-$ approximation, we truncate
\begin{align}
  \widetilde{B}_r=QV_r\Lambda_rV_r^*Q^*  
\end{align}
Clearly, $\lVert PBP-\widetilde{B}_r\rVert_2=\lambda_{r+1}(PBP)\leq \lambda_{r+1}(B)=s^2_{r+1}$ where $\lambda_{r+1}(M)$ denotes the $r+1-$th ordered eigenvalue of the matrix $M$. Thus we get
\begin{align}\label{eq:App2}
    \lVert B-\widetilde{B}_r\rVert_2&\leq \lVert B-PBP\rVert_2+\lVert PBP-\widetilde{B}_r\rVert_2 \nonumber\\
    &\leq 2\lVert (I-P)B\rVert_2+s_{r+1}^2.
\end{align}
Thus if we define, $\beta_{r,p}:=2\alpha_{r,p}+s_{r+1}^2$, and $E_{r,p}=\{\lVert B-\widetilde{B}\rVert_2\leq \beta_{r,p}\}$ then from \eqref{eq:App1} and \eqref{eq:App2} we get
    $\mathbb{P}(E_{r,p})\geq 1-3e^{-p}.$ Finally, on the event $E_{r,p}$, it follows from \eqref{rand_err} that,
    \begin{align*}
        \lVert \Gamma_K-\widetilde{\Gamma}_{K,r}\rVert_2\leq \sigma^2 \omega^2 H_{K,\rho}\beta_{r,p}:=\Delta_{K,r,p}
    \end{align*}
Thus, $\mathbb{P}_{\mathrm{rSVD}}(\lVert \Gamma_K -\widetilde{\Gamma}_{K,r}\rVert_2\leq \Delta_{K,r,p})\geq 1-3 e^{-p}.$ 

% \seanc{This is really the end of the proof, and the rest of the text is remarks (some of which is already in Remark \ref{rem2}). I'd move the rest of the text to after the proof and combine with the remark already there.} 
% \ac{removed it and added more recent references in the remark.}

% Note that for sufficiently large $r$, $\beta_{r,p}$ and consequently $\Delta_{K,r,p}$ is small. In practice, the residual norm $\|B-\widetilde B_r\|_2$ may instead be estimated algorithmically using other randomized techniques and used to obtain a computable approximation to $\Delta{K,r,p}$.

\begin{remark}\label{rem2}
 For a chosen tolerance level $\epsilon$, the rank $r$ may in fact adaptively be chosen to ensure $\beta_{r,p}\leq \epsilon$ with high probability by employing \cite [Algorithm 4.2]{halko2011} in lieu of \cite[Algorithm 4.1]{halko2011} at some extra one-time computational cost. Additionally, more recent developments in randomized numerical linear algebra allow one to employ `single-pass' algorithms, such as ``Nystr\"om method" to more efficiently compute the randomised PSD approximation to PSD matrices directly, see e.g. \cite[section 8]{tropp2023}. We also note that the singular values of the Radon transform decay like $j^{-1/2}$, and so the bound in terms of $\sum_{j>r} s_j^4$ is expected to decay as $r \rightarrow \infty$ independent of discretization.
\end{remark}

\end{proof}

\bibliographystyle{abbrv}
\bibliography{refs}

\end{document}